\documentclass{birkjour}
\usepackage{latexsym}
\usepackage{amssymb}
\usepackage{euscript}
\usepackage{graphics}
\usepackage{comment}
\usepackage{setspace}
\usepackage[foot]{amsaddr}

\newcommand{\res}{\mathrm{res}}
\newcommand{\qekv}{\quad \Longleftrightarrow \quad}
\newcommand{\dil}{\mathrm{dil}}

\def\uphar{{\upharpoonright}}
    
\def\sD{{\mathfrak D}}

\def\dD{{\mathbb D}}

   \def\dN{{\mathbb N}}   
      
   \def\dT{{\mathbb T}}

\def\cD{{\mathcal D}}      
   \def\cH{{\mathcal H}}   
   \def\cK{{\mathcal K}}   \def\cL{{\mathcal L}}
   \def\cN{{\mathcal N}}   
      
      \def\cU{{\mathcal U}}
      \def\cX{{\mathcal X}}
\def\cY{{\mathcal Y}}

\def\bS{{\mathbf S}}

\def\d1{{\mathcal D}}

\def\bm\chi{\mbox{\boldmath$\chi$}}

\def\ran{{\rm ran\,}}

\let\xker=\ker \def\ker{{\xker\,}}
\def\cspan{{\rm \overline{span}\, }}

\def\br{\bigr}

\def\d1{{\mathcal D}}

\def\SK{{\bS_{\kappa}( \cU,\cY)}}
\def\So{{\bS( \cU,\cY)}}
\def\T{{\T_{\Sigma}}}
	
 \newcommand{\Sys}{\Sigma=(T_{\Sigma};\cX,\cU,\cY;\kappa)}
 \newcommand{\sys}{\Sigma=(A,B,C,D;\cX,\cU,\cY;\kappa)}

\newtheorem{theorem}{Theorem}[section]
\newtheorem{lemma}[theorem]{Lemma}
\newtheorem{proposition}[theorem]{Proposition}

\numberwithin{equation}{section}
\theoremstyle{definition}
\newtheorem{definition}[theorem]{Definition}
\newtheorem{remark}[theorem]{Remark}
\newtheorem{example}[theorem]{Example}
\newtheorem*{remark*}{Remark}

\begin{document}

%
%
%
%
%
%
%
%
%

\title[Minimal realizations]
{Minimal passive realizations of generalized Schur functions in Pontryagin spaces}

\author[L. Lilleberg]{Lassi Lilleberg}

\address{Department of Mathematics and Statistics\\
University of Vaasa\\ P.O. Box 700\\ 65101 Vaasa\\
Finland}

\email{lassi.lilleberg@uva.fi}




\begin{abstract}
Passive discrete-time systems in Pontryagin space setting are investigated. In this case the transfer functions of passive systems, or characteristic functions of  contractive operator colligations, are generalized Schur functions.
The existence of optimal and $^*$-optimal minimal realizations for generalized Schur functions are proved. By using those realizations, a
new definition, which covers the case of generalized Schur functions, is given for defects functions. A criterion due to D.Z. Arov and M.A. Nudelman, when all  minimal passive realizations of the same Schur function  are unitarily similar, is generalized to the class of generalized Schur functions. The approach used here is new; it relies completely on the theory of passive systems.
\end{abstract}

\maketitle


\section{Introduction}
 An \textbf{operator colligation} $\Sigma=(T_{\Sigma};\cX,\cU,\cY;\kappa)$
   consists of separable Pontryagin spaces $\cX$ 
    (the \textbf{state space}),
   $\cU$  (the \textbf{incoming space}), and $\cY$ (the \textbf{outgoing space})
   and the  \textbf{system operator} $T_{\Sigma} \in   \cL(\cX\oplus\cU,\cX\oplus\cY),$
                                                                                                                                                                                  the space of bounded operators from $\cX\oplus\cU $  to $\cX\oplus\cY,$ where $\cX \oplus \cU,$ or    { \tiny $\begin{pmatrix}
                                                                                                                                                                                     \cX \\
                                                                                                                                                                                     \cU
                                                                                                                                                                                   \end{pmatrix} 
                                                                                                                                                                                   $}, means the direct orthogonal sum with respect to the indefinite inner product.
  The symbol $\kappa$ is
 reserved for the finite negative index of the state space.   The operator $T_{\Sigma}$
   has the block representation of the form
\begin{equation}\label{colli}
T_{\Sigma}=\begin{pmatrix} A&B \cr C&D\end{pmatrix} :
\begin{pmatrix} \cX \\ \cU \end{pmatrix} \to
\begin{pmatrix} \cX \\ \cY \end{pmatrix},
\end{equation} where $A \in \cL(\cX)$ (the \textbf{main operator}), $B \in \cL(\cU,\cX)$ (the \textbf{control operator}),
 $C \in \cL(\cX,\cY)$ (the \textbf{observation
operator}), and $D \in \cL(\cU,\cY)$ (the \textbf{feedthrough operator}). If needed,  the colligation is  written as
 $\Sigma=(A,B,C,D;\cX,\cU,\cY;\kappa).$  It is always assumed in this paper that $\cU$ and $\cY$ have the same negative index.
 All  notions of continuity and convergence are understood
to be with respect to the strong topology, which is induced by any fundamental decomposition of the space in question. 

The colligation \eqref{colli} will be called as a \textbf{system}
 since it can be seen as a \textbf{linear discrete time system}
of the form
\begin{equation*}
\begin{cases}
 h_{k+1} &=Ah_k+B\xi_k, \\
  \sigma_{k} &=Ch_k+D\xi_k,
\end{cases}\quad k\geq 0,
\end{equation*}
where $\{h_k\}\subset \cX$, $\{\xi_k\}\subset \cU$ and
$\{\sigma_k\}\subset \cY.$ In what follows, the "system" is identified with the operator expression appearing in \eqref{colli}. 
 When the system operator $T_{\Sigma}$ in \eqref{colli}
is  contractive (isometric, co-isometric, unitary), with respect to the indefinite inner product,  the corresponding  system is called
\textbf{passive} (isometric, co-isometric, conservative).
 In literature, conservative systems are also called unitary systems.  The \textbf{transfer function }of the system \eqref{colli}
is defined by
\begin{equation}
\label{trans} \theta_\Sigma(z):=D+z C(I-z
A)^{-1}B,
\end{equation} whenever $I-z
A$ is invertible. Especially, $\theta_\Sigma$ is defined and holomorphic in a neighbourhood of the origin.   The values $\theta_{\Sigma}(z)$ are bounded operators from $\cU$ to $\cY.$ Conversely, if $\theta$ is an operator valued function  holomorphic in a neighbourhood of the origin, and transfer function of the system $\Sigma$ coinsides with it, then $\Sigma$ is a \textbf{realization} of $\theta.$ In some sources,
transfer functions of the systems are also called characteristic functions of operator colligations.

The \textbf{adjoint} or \textbf{dual} of the system $\Sigma$ is the system $\Sigma^*$ such that its system operator is the indefinite  adjoint $T_{\Sigma}^*$  of $T_\Sigma.$ That is,  $\Sigma^*=(T_{\Sigma}^*;\cX,\cY,\cU;\kappa).$ In this paper, all the adjoints are with respect to the indefinite inner product.  For an operator valued function $\varphi,$  the notation ${\varphi}^*({z})$ is used instead of $\left(\varphi({z})\right)^*, $ and the function $\varphi^\#(z)$ is defined to be $\varphi^*(\bar{z}).$  With this notation, for the transfer function $\theta_{\Sigma^*}$ of $\Sigma^*,$ it  clearly holds  $ \theta_{\Sigma^*}(z)={\theta_{\Sigma}}^\#(z).   $
 Since  contractions between  Pontryagin spaces with the same negative index are bi-contractions (cf. eg. \cite[Corollary 2.5]{rovdrit}), $\Sigma^*$ is passive whenever $\Sigma$ is.

 In the case where all the spaces are Hilbert spaces, it is well known; see for instance \cite[Proposition 8]{Arov}, that the transfer function of the passive system is a \textbf{Schur function.} That is, contractive operator valued function  holomorphic in the unit disk $\dD$. In the case where $\cU$ and $\cY$ are Hilbert spaces and the state space $\cX$ is a Pontryagin space, Saprikin showed in \cite[Theorem 2.2]{Saprikin1} that  the transfer function of the passive system  \eqref{colli} is a \textbf{generalized Schur function.}
 It will be proved later in Proposition \ref{maxdim}  that this result holds also in the case when all the spaces are Pontryagin spaces.
  The \textbf{generalized Schur class} $\bS_{\kappa}( \cU,\cY)$, where $\cU$ and $\cY$ are Pontryagin spaces with the same negative index,  is the set of 
 $\mathcal{L}(\cU,\cY)$-valued functions $S(z)$ 
  holomorphic in a neighbourhood $\Omega$ of the origin such that the Schur kernel \begin{equation}\label{kernel1}
K_S(w,z)=\frac{1-S(z)S^*(w)}{1-z\bar{w}}, \qquad w,z \in \Omega,
\end{equation} has $\kappa$ negative squares ($\kappa=0,1,2,\ldots$). This means that for any finite set of points $w_1,\ldots,w_n$
 in the domain of holomorphy $\rho(S)$ of $S$ and set of vectors $\{f_1,\ldots,f_n\} \subset \cY,$ the Hermitian matrix
\begin{equation}\label{kernelmatrix1}
\left(\left\langle K_S(w_j,w_i)f_j,f_i \right\rangle_{\cY}\right)_{i,j=1}^{n},
\end{equation} where $\left\langle\cdot, \cdot\right\rangle_{\cY}$ is the indefinite inner product of the space $\cY,$
 has no more than $\kappa$ negative eigenvalues, and there exists at least one such matrix
 that has exactly $\kappa$ negative eigenvalues. A function $S$ belongs to $\SK$ if and only if $S_\kappa^\# \in \mathbf{S}(\cY,\cU)$; see \cite[Theorem 2.5.2]{ADRS}. The class $\mathbf{S}_0(\cU,\cY)$ coinsides with the ordinary Schur class, and it is written  as $\mathbf{S}(\cU,\cY).$

 The direct connection between the transfer functions of  passive systems of the form \eqref{colli} and the generalized Schur functions allows to study the properties of  generalized Schur functions by using  passive systems, and vice versa. Therefore, a fundamental problem of the subject is, for a given $\theta \in \SK,$ find a realization $\Sigma$ of $\theta$ with the desired minimality or optimality properties (observable, controllable, simple, minimal, optimal,  $^*$-optimal); for details, see Theorems \ref{realz} and \ref{optimals} and Lemma \ref{repres}. The described problem is called a \textbf{realization problem.} In the standard Hilbert space settings, realizations problems, as well as other properties of  passive systems, were studied, for instance, by Ando \cite{Ando},   Arov \cite{A,Arov},
     Arov et al. \cite{ArKaaP,ArKaaP3,ArNu1,ArNu2},
de Branges and Rovnyak \cite{BrR1,BrR2},   Brodski\u{\i} \cite{Br1},  Helton \cite{Helton} and Sz.-Nagy and Foias \cite{SF}. The case where the state space is a Pontryagin space while incoming and outgoing spaces are still Hilbert spaces, unitary systems were studied, for instance, by  Dijksma et al. \cite{DLS1,DLS2}, and passive systems by Saprikin \cite{Saprikin1},
Saprikin and Arov  \cite{SapAr}, Saprikin et al.  \cite{Saprikinetal} and by the author  in \cite{Lassi}. The case where all the spaces are Pontryagin spaces, theory of isometric, co-isometric and conservative systems is  considered, for instance,  in \cite{ADRS,sym-col,rovdrit}.

Especially, in   \cite{Arov}, Arov proved the existence of so-called optimal minimal realizations of an ordinary Schur function; for definitions, see Section \ref{sec-opmin}. The proof was based on the existence (right) \textbf{defect functions}. For an ordinary Schur function $S(\zeta)$, the (right) defect function $\varphi$ of $S$ is, roughly speaking, the maximal analytic minorant of $I-S^*(\zeta)S(\zeta).$ More precicely, this means that for almost every (a.e.) $\zeta$ on the unit circle $\dT,$ it holds
$$ \varphi^*(\zeta) \varphi (\zeta) \leq    I-S^*(\zeta)S(\zeta),  $$ and for every other operator valued analytic function $\widehat{\varphi}$ with similar property, it holds
$$ \widehat{\varphi}^*(\zeta) \widehat{\varphi} (\zeta) \leq  \varphi^*(\zeta) \varphi (\zeta)                 .  $$ For the existence of  defect functions, see \cite[Theorem V.4.2]{SF}, and for a detailed treatise, see \cite{Boiko97, Boikop1, Boikop2}. In \cite{ArKaaP}, Arov et al. constructed   ($^*$-)optimal minimal passive systems in the Hilbert space setting without using  defect functions. The construction can be done by taking an appropriate   restriction of some system. In the indefinite setting, if one uses a suitable definition of optimality, a similar method as was used by Arov et al. still produces a  ($^*$-)optimal minimal passive system. In Pontryagin state space case, this was proved by Saprikin \cite{Saprikin1}. It will be shown in Theorem \ref{optimals} that the same result still holds in the case where all the spaces are Pontryagin spaces.

 The study of the class of generalized Schur functions $\SK$    was continued in \cite{Saprikinetal,SapAr}, in the case where $\cU$ and $\cY$ are Hilbert spaces and the state space is a Pontryagin space.
 In \cite{SapAr}, Saprikin and Arov used  the right Kre\u{\i}n-Langer factorization of the form  $S=S_rB_r^{-1}$ for $S \in \SK,$
 and proved that
 the existence of the optimal minimal realization of $S$  is equivalent to the existence of the right defect function of $S_r.$ However, they did not define the defect functions for the generalized Schur functions. This was done by the author in \cite{Lassi} by using the Kre\u{\i}n-Langer factorizations. With the definition given therein, the main results of \cite{Seppo} were generalized to the Pontryagin state space setting. The main subjects of \cite{Lassi} include some continuation of the study of products of  systems and the stability properties of passive systems,  subjects  treated earlier by Saprikin et al. in \cite{Saprikinetal}.   In the present paper, it will be shown that a concept of defect functions can be defined in the case where all the spaces are Pontryagin spaces. The key idea here is to use optimal minimal passive realizations and conservative embeddings. By using such a  definition, it is shown that one can generalize and improve some  of the main results from \cite{Seppo}, using different  proofs than those given in  \cite{Seppo} or \cite{Lassi}, see Theorem \ref{sepontulos}. Furthermore, in Theorem \ref{kulmaminimal},  the main results from \cite{ArNu1,ArNu2} concerning  the criterion when all the minimal realizations of a Schur function are unitarily similar, is generalized to the present indefinite setting. The proof will be carried out entirely by using the theory of passive systems, without applying   Hardy space theory or the theory of Hankel operators as  in the proof provided in \cite{ArNu2}.

 The paper is organized as follows. In Section \ref{dil-em} basic facts of linear systems, Julia operators, dilations and embeddings are recalled. As a preparation, it is shown in Proposition \ref{consdil} that an arbitrary passive linear system has a conservative dilation. Moreover, Lemma \ref{repres} gives some usefull representations and restrictions of passive systems. That lemma will be used extensively later on in this paper.

 In Section \ref{sec-opmin}, the existence and basic properties of ($^*$-)optimal minimal realizations are established. The main result of this section is Theorem \ref{optimals}.

   The generalized defect functions are introduced in Section \ref{sec7}. In particularly,  Theorem \ref{kulmaminimal}   in this section can be seen as the main result of the paper.

  \section{Linear systems, dilations and embeddings} \label{dil-em}

Let $\Sys$ be a linear system as in \eqref{colli}.
The following subspaces \begin{align}
\cX^c&:=\cspan\{\ran A^{n}B:\,n=0,1,\ldots\} \label{cont1}  \\
\cX^o&:= \cspan\{\ran A^{*n}C^*:\,n=0,1,\ldots\}\label{obs1}  \\
\cX^{s}&:= \cspan\{\ran A^{n}B,\ran A^{*m}C^*:\,n,m=0,1,\ldots\},  \label{simp1}
\end{align} 
 are called, respectively, controllable,  observable and simple subspaces.
The system is said to be \textbf{controllable} (\textbf{observable}, \textbf{simple})
  if $\cX^c=\cX (\cX^o=\cX,\cX^s=\cX)$ and \textbf{minimal} if it is both controllable and observable.
 When $\Omega\ni0$ is some symmetric neighbourhood of the origin, that is, $\bar{z} \in \Omega$ whenever $z\in \Omega,$ then also
 \begin{align}
\cX^c&=\cspan\{\ran (I-zA)^{-1}B, z \in \Omega\} \label{cont2}  \\
\cX^o&=\cspan\{\ran (I-zA^*)^{-1}C^*, z \in \Omega\}\label{obs2}  \\
\cX^{s}&=\cspan\{\ran (I-zA)^{-1}B,\ran (I-wA^*)^{-1}C^*, z,w \in \Omega\} \label{simp2}
\end{align}


  The system \eqref{colli} can be expanded to a larger system  without changing the transfer function. It can be done by using the so-called \textbf{Julia operator}, see \eqref{Juliaoperator} below.
          For a proof of the following theorem  and more details about the Julia operators, see \cite{rovdrit}. The basic information about the indefinite inner product spaces and their operators can be recalled from \cite{Azizov,Bognar,rovdrit}.

          \begin{theorem}\label{Julia} Suppose that $\cX_1$  and $\cX_2$ are Pontryagin spaces with the same negative index,
           and let $A: \cX_1 \to \cX_2$ be a contraction. Then there exist Hilbert spaces $\sD_{A}$
            and $\sD_{A^*},$ linear operators $D_A: \sD_{A}\to\cX_1, D_{A^*}: \sD_{A^*} \to \cX_2 $
             with zero kernels  and a linear operator $L:\sD_{A} \to \sD_{A^*}  $ such that it holds
             \begin{equation}\label{defects}
             I-A^*A=D_{A}D_{A}^*, \qquad  I-AA^*=D_{A^*}D_{A^*}^*,
             \end{equation} and the operator
              \begin{equation}\label{Juliaoperator}
                       U_A:= \begin{pmatrix}
                               A & D_{A^*} \\
                                D^*_{A} & -L^*
                             \end{pmatrix}: \begin{pmatrix}
                                              \cX_1 \\
                                              \sD_{A^*}
                                            \end{pmatrix} \to \begin{pmatrix}
                                             \cX_2 \\
                                              \sD_{A}
                                            \end{pmatrix}
                     \end{equation} is unitary. Moreover, $D_{A}, D_{A^*}$ and $U_A$ are 
                      unique up to unitary equivalence.
          \end{theorem}
\begin{remark}\label{def-def}
       The operator $D_{A}$ from Theorem \ref{Julia} is called a \textbf{defect operator}  of $A.$
\end{remark}

  A \textbf{ dilation} of a system $\sys$ is
           any system of the form $\widehat{\Sigma}=(\widehat{A},\widehat{B},\widehat{C},D;\widehat{\cX},\cU,\cY;\kappa),$ where
  \begin{equation}\label{dilation}
 \widehat{\cX}=  \cD\oplus{\cX}\oplus\cD_*  , \quad \widehat{A} \cD \subset \cD, \quad \widehat{A}^* \cD_* \subset \cD_*, \quad \widehat{C} \cD =\{0\}, \quad \widehat{B}^* \cD_* =\{0\}.
 \end{equation} The spaces $\cD$ and $\cD_*$ are required to be Hilbert spaces.  The system operator  $T_{\widehat{\Sigma}}$ of $\widehat{\Sigma}$ is of the form
  \begin{equation}\label{dilatio-blok}\begin{split}
    T_{\widehat{\Sigma}}=\begin{pmatrix}
                           \begin{pmatrix}
                             A_{11} & A_{12} & A_{13} \\
                             0 & A & A_{23} \\
                             0 & 0 & A_{33}
                           \end{pmatrix} & \begin{pmatrix}
                                             B_1\\
                                            B \\
                                            0
                                          \end{pmatrix} \\
                           \begin{pmatrix}
                                             0&
                                            C &
                                            C_1
                                          \end{pmatrix} & D
                         \end{pmatrix}: \begin{pmatrix}
                                          \begin{pmatrix}
                                            \cD \\
                                            \cX \\
                                            \cD_*
                                          \end{pmatrix} \\
                                          \cU
                                        \end{pmatrix} \to\begin{pmatrix}
                                          \begin{pmatrix}
                                            \cD \\
                                            \cX \\
                                            \cD_*
                                          \end{pmatrix} \\
                                          \cY
                                        \end{pmatrix}, \\
                                        \widehat{A}=\begin{pmatrix}
                             A_{11} & A_{12} & A_{13} \\
                             0 & A & A_{23} \\
                             0 & 0 & A_{33}
                           \end{pmatrix} ,\qquad\widehat{B}=\begin{pmatrix}
                                             B_1\\
                                            B \\
                                            0
                                          \end{pmatrix},\qquad\widehat{C}= \begin{pmatrix}
                                             0&
                                            C &
                                            C_1
                                          \end{pmatrix}. \end{split}
  \end{equation}

  The system $\Sigma$ is called a \textbf{restriction} of $\widehat{\Sigma}.$
   Recall that  subspace $\cN$ of the Pontryagin space $\cH$ is  \textbf{regular} if it is itself a Pontryagin space with the inherited inner product of $\langle \cdot ,\cdot\rangle_{\cH}.$ The subspace $\cN$ is regular precicely when $\cN^\perp$ is regular, where  $\perp$ refers to orthogonality with respect to the indefinite inner product of $\cH$.
   Since $\cX$ clearly is  a {regular} subspace of $\widehat{\cX},$ there exists the unique orthogonal projection $P_{\cX}$ from $\widehat{\cX}$ to $\cX.$ Let $\widehat{A}\uphar_{\cX}$ be the restriction of $\widehat{A}$ to the subspace $\cX.$ Then, the system $\Sigma$  can be represented as \begin{equation}\label{restri} \Sigma=(P_{\cX}\widehat{A}\uphar_{\cX},P_{\cX}\widehat{B},\widehat{C}\uphar_{\cX},D;P_{\cX}\widehat{\cX},\cU,\cY;\kappa).\end{equation}
  Dilations and restrictions are denoted by
  \begin{equation}\label{dil-res}
 \widehat{\Sigma}=\dil_{\cX \to \widehat{\cX}} \Sigma, \quad \Sigma=\res_{ \widehat{\cX} \to\cX} \widehat{\Sigma},
 \end{equation} mostly without subscripts when the corresponding state spaces are clear.
  A calculation show that the transfer functions of the original system and its dilation coincide. Moreover, if $\Sigma$ is passive, then is any retriction of it.  The following proposition shows that a passive system  has a conservative dilation.
   For the Hilbert space case, this result is from \cite{A}, and for the Pontryagin state space case, see \cite{Saprikin1}.

  \begin{proposition}\label{consdil}
    Let $\sys$ be a passive system. Then there exists a conservative dilation $\widehat{\Sigma}=(\widehat{A},\widehat{B},\widehat{C},D;\widehat{\cX},\cU,\cY;\kappa)$  of $\Sigma$.
  \end{proposition}
  \begin{proof}
    Since the system operator $T$ of $\Sigma$ is a contraction, by Theorem \ref{Julia}, there exists a Julia operator
   \begin{equation}\begin{split} \label{JuliaoperatorT}
                   \begin{pmatrix}
                               T & D_{T^*} \\
                                D^*_{T} & -L^*
                             \end{pmatrix}&: \begin{pmatrix}
                                              \cX \oplus \cU \\
                                              \sD_{T^*}
                                            \end{pmatrix} \to \begin{pmatrix}
                                             \cX \oplus \cY\\
                                              \sD_{T}
                                            \end{pmatrix}  \qekv \\     \begin{pmatrix}
                               \begin{pmatrix}
                                 A & B \\
                                 C & D
                               \end{pmatrix} &\begin{pmatrix}
                                                P_{\cX} D_{T^*} \\
                                                 P_{\cY} D_{T^*}
                                              \end{pmatrix} \\
                               \begin{pmatrix}
                                  D^*_{T}\uphar_{\cX} &  D^*_{T}\uphar_{\cU}
                               \end{pmatrix} & -L^*
                             \end{pmatrix}&: \begin{pmatrix}\begin{pmatrix}
                                                             \cX \\
                                                            \cU
                                                           \end{pmatrix}
                                            \\
                                              \sD_{T^*}
                                            \end{pmatrix} \to \begin{pmatrix}
                                            \begin{pmatrix}
                                                             \cX \\
                                                            \cY
                                                           \end{pmatrix}\\
                                              \sD_{T}
                                            \end{pmatrix}
                  \end{split}   \end{equation} with properties introduced in Theorem \ref{Julia}. Denote $E_{\cH}(h)=\langle h,h \rangle_{\cH}$ for a vector $h$ in an inner product space $\cH.$ Then, for $x \in \cX, u \in \cU$ and $f \in \sD_{T^*},$ one has
                     \begin{equation}\label{unidil}\begin{split}
                       E_{\cX}(x)+E_{\cU}(u)+E_{\sD_{T^*}}(f)&= E_{\cX}(Ax + Bu+P_{\cX}D_{T^*} f)\\&\quad+E_{\cY}(Cx + Du+P_{\cY}D_{T^*} f)\\&\qquad+E_{  \sD_{T}}(D^*_{T} \uphar_{\cX} x + D^*_{T} \uphar_{\cU}u -L^* f).
                     \end{split}\end{equation} Also, for given $x' \in \cX, y \in \cY$ and $f'\in \sD_{T}, $ there exists an unique triplet  $x \in \cX, u \in \cU$ and $f \in \sD_{T^*}$ such that
                                                                 \begin{equation}\label{unirep}
                               \begin{pmatrix}
                                 A & B &  P_{\cX} D_{T^*}\\
                                 C & D &   P_{\cY} D_{T^*}\\
                                  D^*_{T}\uphar_{\cX} &  D^*_{T}\uphar_{\cU}
                                & -L^*
                               \end{pmatrix} \begin{pmatrix}
                                                             x \\
                                                            u

                                            \\
                                              f
                                            \end{pmatrix} =\begin{pmatrix}

                                                             x'\\
                                                            y
                                                      \\
                                              f'
                                            \end{pmatrix}
                                                                 \end{equation}

                   Define $$ \ell_-^2( \sD_{T}):= \bigoplus_{-\infty}^{-1} \sD_{T}, \quad\ell_+^2( \sD_{T^*}):= \bigoplus_{1}^{\infty} \sD_{T^*}, \quad \widehat{\cX}:=  \ell_-^2( \sD_{T}) \oplus \cX \oplus \ell_+^2( \sD_{T^*}),$$ where $\bigoplus$ denotes an orthogonal sum of Hilbert spaces. Since $\ell_-^2( \sD_{T})$ and  $\ell_+^2( \sD_{T^*})$ are Hilbert spaces, the space $\widehat{\cX}$ clearly is a Pontryagin space with the negative index $\kappa$. For $u \in \cU$ and $(\ldots f_{-1},\underline{x},f_1 \ldots )\in \widehat{\cX},$ where the underlined element belongs to $\cX,$ define the operators
                     \begin{align*}
                       &\widehat{A}(\ldots, f_{-2}, f_{-1},\underline{x},f_1,f_{2},f_3, \ldots ):=\\& \qquad(\ldots, f_{-1},D^*_{T} \uphar_ \cX x -L^*f_1, \underline{Ax + P_{\cX}D_{T^*} f_{1}},f_{2}, f_3,\ldots ),
                        \\  &\widehat{B}u:=(\ldots,  0,  D^*_{T} \uphar_ \cU u, \underline{Bu}, 0,  \ldots) ,\\
                       &\widehat{C}(\ldots,  f_{-1},\underline{x},f_1, \ldots):=Cx + P_{\cY}D_{T^*} f_{1}.
                     \end{align*} A calculation shows that the system $\widehat{\Sigma}:=(\widehat{A},\widehat{B},\widehat{C},D;\widehat{\cX},\cU,\cY;\kappa)$ is a dilation of $\Sigma$.
                       Moreover, calculations by using the identities \eqref{unidil} and \eqref{unirep} show that the system operator $T_{\widehat{\Sigma}}$ is unitary, and therefore the system $\widehat{\Sigma}$ is conservative.
  \end{proof}
It is possible that $\cD=\{0\}$ or $\cD_*=\{0\}$ in \eqref{dilation}. In those cases, the zero space and the corresponding row and column will be left out in \eqref{dilatio-blok}. In particular,
 if the system $\Sigma$ with the system operator $T$ as in \eqref{colli} is isometric (co-isometric), then $D_T=0$ ($D_{T^*}=0$), and proceeding as in the proof of Proposition \ref{consdil}, it is possible to construct a conservative dilation $\widehat{\Sigma}$ of the form \eqref{dilatio-blok} such that $\cD=\{0\}$ ($\cD_*=\{0\}$). Moreover, for an arbitrary passive system as in \eqref{colli}, it is possible to construct an isometric (co-isometric) dilation $\widehat{\Sigma}$ of the form \eqref{dilatio-blok} such that $\cD_*=\{0\}$ ($\cD=\{0\}$).

  There is also an another  way to expand the  system \eqref{colli}, and it is called an  \textbf{embedding}. In this expansion, the state space and the main operator will not change. The embedding of  the  system \eqref{colli}   is
  any system determined by the system operator 
 \begin{equation}\label{embed}\begin{split}
   T_{\widetilde{\Sigma}}=\begin{pmatrix}
                             A & \widetilde{B} \\
                            \widetilde{C} & \widetilde{D}
                           \end{pmatrix}&: \begin{pmatrix}
                              \cX     \\  \widetilde{\cU}
                                 \end{pmatrix} \to \begin{pmatrix}
                                 \cX \\  \widetilde{\cY}
                                 \end{pmatrix} \qekv \\   \begin{pmatrix}
                             A & \begin{pmatrix}
                                   B & B_1
                                 \end{pmatrix} \\
                           \begin{pmatrix}
                                   C \\ C_1
                                 \end{pmatrix} & \begin{pmatrix}
                                                   D & D_{12} \\
                                                   D_{21} & D_{22}
                                                 \end{pmatrix}
                           \end{pmatrix}&: \begin{pmatrix}
                                            \cX \\
                                            \begin{pmatrix}
                                              \cU \\
                                              \cU'
                                            \end{pmatrix}
                                          \end{pmatrix}\to \begin{pmatrix}
                                            \cX \\
                                             \begin{pmatrix}
                                              \cY \\
                                              \cY'
                                            \end{pmatrix}
                                          \end{pmatrix},
  \end{split}\end{equation}  where $\cU'$ and $\cY'$ are Hilbert spaces.
 The transfer function of the embedded system is \begin{equation}\label{transemb}\begin{split}
                                                             \theta_{\widetilde{\Sigma}}(z)&= \begin{pmatrix}
                                                   D & D_{12} \\
                                                   D_{21} & D_{22}
                                                 \end{pmatrix} +z  \begin{pmatrix}
                                   C \\ C_1
                                 \end{pmatrix}(I_{\cX}-zA)^{-1}\begin{pmatrix}
                                   B & B_1
                                 \end{pmatrix} \\ &= \begin{pmatrix}
                                                       D + zC(I_{\cX}-zA)^{-1}B & D_{12}+zC(I_{\cX}-zA)^{-1}B_1 \\
                                                          D_{21} +zC_1(I_{\cX}-zA)^{-1}B & D_{22} + zC_1(I_{\cX}-zA)^{-1}B_1
                                                     \end{pmatrix} \\&=
                                                             \begin{pmatrix}
                                                                                           \theta_{\Sigma} (z) & \theta_{12}(z)\\
                                                                                            \theta_{21}(z) &  \theta_{22}(z)
                                                                                         \end{pmatrix},
                                                           \end{split}\end{equation}
                                                            where $\theta_{\Sigma}$ is the transfer function of the original system. The embedded systems will be needed in Section  \ref{sec7}.

It will be proved in Proposition \ref{maxdim} below that the transfer function of  any passive system  \eqref{colli} is a generalized  Schur function with index not larger than the negative index of the state space. For a special case where incoming and outcoming spaces are Hilbert spaces, this result is due to \cite[Theorem 2.2]{Saprikin1}.  The proof of the general case follows the lines of Saprikin's proof of the special case.
\begin{lemma}\label{ker-es}
  Let $\Sigma=(A,B,C,D;\cX,\cU,\cY;\kappa)$ be a passive system with the transfer function $\theta$. Denote the system operator of $\Sigma$ as $T.$   If   \begin{align}\label{Deet1} D_{T} = \begin{pmatrix}
                                      D_{T_{,1}}\\
                                      D_{T_{,2}}
                                    \end{pmatrix}:\cD_{T} \to \begin{pmatrix}
                                                                 \cX \\
                                                                 \cU
                                                               \end{pmatrix}\quad
  D_{T^*} = \begin{pmatrix}
                                      D_{T_{,1}^*}\\
                                      D_{T_{,2}^*}
                                    \end{pmatrix}: \cD_{T^*} \to \begin{pmatrix}
                                                                 \cX \\
                                                                 \cY
                                                               \end{pmatrix},
                                      \end{align} are defect operators of $T$ and $T^*$, respectively, then the identities \begin{align}\label{lefdef2}
 I_{\cY}- \theta(z) \theta^*(w)&=(1 -z\bar{w}) G(z)G^*(w)  +\psi(z) \psi^*(w) ,\\
     I_{\cU}-\theta^* (w) \theta(z)&=(1 -z\bar{w}) F^*(w)F(z)  +\varphi^*(w) \varphi(z),\label{lefdef3}
\end{align}
  with \begin{align}G(z)&=C(I_{\cX}-zA)^{-1},&   \quad\psi(z)&= D_{T_{,2}^*}+zC(I_{\cX}-zA)^{-1} D_{T^*_1}& \quad ,\label{gees}\\
   F(z)&=(I_{\cX}-zA)^{-1}B,&   \quad \varphi(z)&= D_{T_{,2}}^* +zD_{T_{,1}}^*(I_{\cX}-zA)^{-1}B,&\label{efs}\end{align} hold for every $z$ and $w$ in a sufficiently small symmetric neighbourhood  of the origin. 
\end{lemma}
\begin{proof}
   By applying the results from \cite[Theorem 1.2.4]{ADRS} and the identities in \eqref{defects}, one deduces that for every $z$ and $w$ in a sufficiently small symmetric neighbourhood of the origin, it holds
    \begin{align*}
     &I_{\cY}-  \theta(z) \theta^*(w)= \begin{pmatrix}
                    C(I_{\cX}-zA)^{-1} & I_{\cY}
                  \end{pmatrix} \begin{pmatrix}
                   (I_{\cX}-\bar{w}A^*)^{-1}C^*  \\ I_{\cY}
                  \end{pmatrix} \\
   &\quad- \begin{pmatrix}
         zC(I_{\cX}-zA)^{-1}& I_{\cY}\end{pmatrix} 
         TT^* 
          \begin{pmatrix}
                                                         \bar{w}(I_{\cX}-\bar{w}A^*)^{-1}C^* \\ I_{\cY}
                                                       \end{pmatrix} \\
         &= \begin{pmatrix}
                    C(I_{\cX}-zA)^{-1} & I_{\cY}
                  \end{pmatrix} \begin{pmatrix}
                   (I_{\cX}-\bar{w}A^*)^{-1}C^*  \\ I_{\cY}
                  \end{pmatrix} \\
   &\quad- \begin{pmatrix}
         zC(I_{\cX}-zA)^{-1}& I_{\cY}\end{pmatrix} \left(
          I-   D_{T^*}D_{T^*}^*
          \right)
          \begin{pmatrix}
                                                         \bar{w}(I_{\cX}-\bar{w}A^*)^{-1}C^* \\ I_{\cY}
                                                       \end{pmatrix} \\
                                                       &= C(I_{\cX}-zA)^{-1}
                   (I_{\cX}-\bar{w}A^*)^{-1}C^* +I_{\cY}\\
   &\quad- \!\begin{pmatrix}\!
         zC(I_{\cX}-zA)^{-1}& I_{\cY}\!\end{pmatrix}\! \!  \begin{pmatrix}\!
                                                    \bar{w}\! \left(I_{\cX}-D_{T_{,1}^*}D_{T_{,1}^*}^* \right)\!(I_{\cX}-\bar{w}A^*)^{-1}C^* \!-\!D_{T_{,1}^*}D_{T_{,2}}^*   \\
                                                      \! -\bar{w}D_{T_{,2}^*}D_{T_{,1}^*}^* (I_{\cX}-\bar{w}A^*)^{-1}C^*\!+   I_{\cY}\!-\!D_{T_{,2}^*}D_{T_{,2}^*}^*
                                                  \! \!\end{pmatrix}
                                                    \\
                                                       &=(1 -z\bar{w}) C(I_{\cX}-zA)^{-1}
                   (I_{\cX}-\bar{w}A^*)^{-1}C^* \\
                                                  & \quad  +z\bar{w}C(I_{\cX}-zA)^{-1} D_{T_{,1}^*}D_{T_{,1}^*}^*(I_{\cX}-\bar{w}A^*)^{-1}C^* + zC(I_{\cX}-zA)^{-1}D_{T_{,1}^*}D_{T_{,2}}^* \\
                                                   & \quad \quad +\bar{w}D_{T_{,2}^*}D_{T_{,1}^*}^* (I_{\cX}-\bar{w}A^*)^{-1}C^* +D_{T_{,2}^*}D_{T_{,2}^*}^*
                                                   \\
                                                       &=(1 -z\bar{w}) G(z)G^*(w)  +\psi (z) \psi^*(w) .
   \end{align*}  Similar calculations show that \eqref{lefdef3} holds also, and the proof is complete.
\end{proof}
Note that if $\Sigma$ in Lemma \ref{ker-es} is isometric (co-isometric), then $D_T=0$ ($D_{T^*}=0$) and therefore $\varphi\equiv0$ ($\psi\equiv0$).

  \begin{proposition}\label{maxdim}
   If $\Sigma=(A,B,C,D;\cX,\cU,\cY;\kappa)$ is a passive system, the the transfer function  $\theta$ of $\Sigma$ belongs to $ \mathbf{S}_{\kappa'}(\cU,\cY),$ where $\kappa'\leq \kappa.$
 \end{proposition}
 \begin{proof}
   Denote the system operator of $\Sigma$ as $T.$ By Lemma \ref{ker-es}, the kernel $ K_\theta$ defined as in \eqref{kernel1} has a representation \begin{equation}\label{ker-eq} K_\theta(w,z)= G(z)G^*(w)  +(1 -z\bar{w}) ^{-1}\psi (z) \psi^*(w),\end{equation} where $G(z)$ and $\psi(z)$ are defined as in \eqref{gees}. Since the negative index of $\cX$ is $\kappa$ and the negative index of the Hilbert space $\sD_{T^*}$ is zero, it follows from \cite[Lemma 1.1.1.]{ADRS},  that for any finite set of points $w_1,\ldots,w_n$
 in the domain of holomorphy of $\theta$ and the set of vectors $\{y_1,\ldots,y_n \}\subset \cY,$ the Gram matrices
  $$ \left(\!\left\langle G^*(w_j)y_j, G^*(w_i)y_i \right\rangle_{\cX}\!\right)_{i,j=1}^{n}, \qquad \left(\left\langle \psi^*(w_j)y_j,\psi^*(w_i)y_i \right\rangle_{\sD_{T^*}}\right)_{i,j=1}^{n}, $$ have, respectively,  at most $\kappa$ and zero negative eigenvalues. 
        The kernel $(1 -z\bar{w}) ^{-1}$ has no negative square, since it is the reproducing kernel of the classical Hardy space $H^2(\dD).$ 
        The Schur product theorem shows  that the kernel  $(1 -z\bar{w}) ^{-1}\psi (z) \psi^*(w)$ has no negative square. Then it follows  from
\cite[Theorem 1.5.5] {ADRS} that the kernel $K_{\theta}$ has at most $\kappa$ negative square. That is, $ \theta \in \mathbf{S}_{\kappa'}(\cU,\cY),$ where $\kappa'\leq \kappa,$ and the proof is complete.
 \end{proof}
\begin{definition}\label{admis}
   A passive realization $\Sigma$ of a genaralized Schur function $\theta \in \SK$ is called \textbf{$\kappa$-admissible} if the negative index of the state space of $\Sigma$ coinsides with the negative index $\kappa$ of $\theta.$
\end{definition} In what follows, this paper deals mostly with the $\kappa$-admissible realizations.
 It will turn out that the $\kappa$-admissible  realizations of $\theta \in \SK$ are well behaved is some sense; they have
 many similar propeties than the standard passive Hilbert space systems.

  The following realizations theorem is well known,
 %
 %
%
 see  \cite[Theorems 2.2.1, 2.2.2 and 2.3.1]{ADRS}.
\begin{theorem}\label{realz}
  For a generalized Schur function $\theta \in \SK$ there exist realizations $\Sigma_k=(T_k;\cX_k,\cU,\cY;\kappa), k=1,2,3,$ of $\theta$ such that
  \begin{itemize}
    \item[\rm{(i)}] $\Sigma_1$ is  observable co-isometric;
    \item[\rm{(ii)}] $\Sigma_2$ is  controllable isometric;
    \item[\rm{(iii)}]$\Sigma_3$ is simple conservative.
  \end{itemize} Conversely, if the system $\Sigma$ has some of the properties {\rm(i)--(iii)},
   then $\theta_{\Sigma}\in \SK,$ where $\kappa$ is the negative index of the state space of $\Sigma.$
 \end{theorem}

 Recall that a \textbf{Hilbert subspace} of the Pontryagin space $\cX$ is a regular subspace such that its negative index is zero. Conversely, \textbf{anti-Hilbert subspace} is a regular subspace such that its positive index is zero.
 When $\cU$ and $\cY$ happens to be Hilbet spaces, the   transfer function $\theta$ of the passive system $\Sys$ belongs to  class $\SK$ (with $\kappa=\mathrm{ind}_{-}\cX$) if and only if $(\cX^s)^\perp$ is a Hilbert subspace \cite[Lemma 3.2]{Lassi}.
  In the case when $\cU$ and $\cY$ are Pontryagin spaces with the same negative index,
  the transfer function $\theta$ of the  isometric (co-isometric, conservative) system $\Sys$
 belongs to class $\SK$  if and only if $(\cX^c)^\perp$ ($(\cX^o)^\perp$,$(\cX^s)^\perp$) is a Hilbert subspace \cite[Theorem 2.1.2]{ADRS}. For a passive system, one has the following  result.
\begin{proposition}\label{OnHilb} For a passive realization $\sys$ of $\theta \in \SK,$ spaces $\cX^c,$ $\cX^o$ and $\cX^s $ are regular and their orthogonal complements are Hilbert subspaces.
\end{proposition}
\begin{proof} Let $\Omega$ be a symmetric neighbourhood of the origin such that $(I-zA)^{-1}$ and $(I-zA^*)^{-1}$ exist for every $z \in \Omega.$
Represent the kernel $K_\theta$ as in \eqref{ker-eq}. Since $K_\theta$ has $\kappa$ negative square, a similar argument used in the proof of \ref{maxdim} shows that the kernel
$      K_1(z,w)=   G(z)G^*(w),$ where $ G(z)=C(I-zA)^{-1}, $ has $\kappa$ negative square. It follows now from \cite[Lemma 1.1.1']{ADRS} that $ \mathrm{span}\{\ran (I-\overline{w}A^*)^{-1}C^*, \overline{w} \in \Omega\}$ contains a $\kappa$-dimensional maximal anti-Hilbert subspace $\cX_\kappa.$  Then, $\cX_\kappa \oplus (\cX_\kappa)^\perp=\cX$ is a fundamental decomposition of $\cX.$ Especially,  $ (\cX_\kappa)^\perp$ is a Hilbert subspace of $\cX.$ But
 $$    \left(    \mathrm{span}\{\ran (I-\overline{w}A^*)^{-1}C^*, \overline{w} \in \Omega\} \right)^\perp  = \left(\cX^o\right)^\perp \subset        (\cX_\kappa)^\perp, $$ which implies that $\left(\cX^o\right)^\perp$ is a Hilbert subspace, and  therefore its orthocomplement $\cX^o$ is regular.

  By duality argument, the space $\cX^c$ is a regular subspace and the space $(\cX^c)^\perp$ is a Hilbert subspace.  It easily follows from \eqref{cont1}--\eqref{simp1} that $(\cX^s)^\perp=(\cX^c)^\perp \cap (\cX^o)^\perp,$ and therefore $(\cX^s)^\perp$ is also a Hilbert subspace and $\cX^s$ is regular.
\end{proof}
It follows from the Proposition \ref{OnHilb} above that the state space $\cX$ of a $\kappa$-admissible realization $\Sigma$ of $\theta \in \SK$ can be decombosed to the controllable, observable and simple parts. Using this fact, the lemma below, which will be used extensively, can be proved.
\begin{lemma}\label{repres}
  Let $\sys$ be a passive system such that the spaces $(\cX^o)^\perp$, $(\cX^c)^\perp$ and $(\cX^s)^\perp$ are Hilbert subspaces of $\cX.$ Then the system operator $T$ of $\Sigma$ has the following representations
{\small  \begin{align}
  T&= \begin{pmatrix}
        \begin{pmatrix}
          A_1 & A_2 \\
          0 & A_o
        \end{pmatrix} &  \begin{pmatrix}
          B_1 \\ B_o
        \end{pmatrix}\\
        \begin{pmatrix}
          0 & C_o
        \end{pmatrix} & D
      \end{pmatrix}: \begin{pmatrix}
                       \begin{pmatrix}(\cX^o)^\perp \\
                       \cX^o\end{pmatrix} \\ \cU
                     \end{pmatrix}   \to \begin{pmatrix}
                       \begin{pmatrix}(\cX^o)^\perp \\
                       \cX^o\end{pmatrix} \\ \cY
                     \end{pmatrix}  \label{rep-obse}\\
  T&= \begin{pmatrix}
        \begin{pmatrix}
          A_3 & 0 \\
          A_4 & A_c
        \end{pmatrix} &  \begin{pmatrix}
          0 \\ B_c
        \end{pmatrix}\\
        \begin{pmatrix}
          C_1 & C_c
        \end{pmatrix} & D
      \end{pmatrix}: \begin{pmatrix}
                       \begin{pmatrix}(\cX^c)^\perp \\
                       \cX^c\end{pmatrix} \\ \cU
                     \end{pmatrix}   \to \begin{pmatrix}
                       \begin{pmatrix}(\cX^c)^\perp \\
                       \cX^c\end{pmatrix} \\ \cY
                     \end{pmatrix} \label{rep-contro}\\
  T&=  \begin{pmatrix}
        \begin{pmatrix}
          A_5 & 0 \\
          0 & A_s
        \end{pmatrix} &  \begin{pmatrix}
          0 \\ B_s
        \end{pmatrix}\\
        \begin{pmatrix}
          0 & C_s
        \end{pmatrix} & D
      \end{pmatrix}: \begin{pmatrix}
                       \begin{pmatrix}(\cX^s)^\perp \\
                       \cX^s\end{pmatrix} \\ \cU
                     \end{pmatrix}   \to \begin{pmatrix}
                       \begin{pmatrix}(\cX^s)^\perp \\
                       \cX^s\end{pmatrix} \\ \cY
                     \end{pmatrix} \label{rep-simp}\\
                     T 
                                                     &=\begin{pmatrix}
                                                                     \begin{pmatrix}A_{11}' & A_{12}' & A_{13}'\\
                                                                     0 & A' & A_{23}'     \\
                                                                     0 & 0 & A_{33}' \end{pmatrix} & \begin{pmatrix}   B_1' \\                     B'   \\ 0   \end{pmatrix} \\
                                                                    \begin{pmatrix} 0 &C' &   C_1' \end{pmatrix} & D
                                                                   \end{pmatrix} \! : \!\begin{pmatrix}
                                                                                    \begin{pmatrix}
                                                                                      (\cX^o)^\perp \\
                                                                                      \overline{P_{\cX^o}\cX^c} \\
                                                                                      \cX^o \cap (\cX^c)^\perp
                                                                                    \end{pmatrix} \\
                                                                                    \cU
                                                                                  \end{pmatrix} \to \begin{pmatrix}
                                                                                   \begin{pmatrix}
                                                                                      (\cX^o)^\perp \\
                                                                                      \overline{P_{\cX^o}\cX^c} \\
                                                                                      \cX^o \cap (\cX^c)^\perp
                                                                                    \end{pmatrix} \\
                                                                                    \cY
                                                                                  \end{pmatrix}\label{rep-fmini}\\
                                                                                   T 
                                                     &=\begin{pmatrix}
                                                                     \begin{pmatrix}A_{11}'' & A_{12}'' & A_{13}''\\
                                                                     0 & A'' & A_{23}''     \\
                                                                     0 & 0 & A_{33}'' \end{pmatrix} & \begin{pmatrix}   B_1'' \\                     B''   \\ 0   \end{pmatrix} \\
                                                                    \begin{pmatrix} 0 &C'' &   C_1'' \end{pmatrix} & D
                                                                   \end{pmatrix}\! : \! \begin{pmatrix}
                                                                                    \begin{pmatrix}
                                                                            \cX^c \cap (\cX^o)^\perp         \\
                                                                                      \overline{P_{\cX^c}\cX^o} \\
                                                                                      (\cX^c)^\perp
                                                                                    \end{pmatrix} \\
                                                                                    \cU
                                                                                  \end{pmatrix} \to \begin{pmatrix}
                                                                                    \begin{pmatrix}
                                                                            \cX^c \cap (\cX^o)^\perp         \\
                                                                                      \overline{P_{\cX^c}\cX^o} \\
                                                                                      (\cX^c)^\perp
                                                                                    \end{pmatrix} \\
                                                                                    \cY
                                                                                  \end{pmatrix}\label{rep-smini}
  \end{align} } The restrictions \begin{align}
  \Sigma_o &=(A_o,B_o,C_o,D;\cX^o,\cU,\cY;\kappa)\label{res-obse} \\
    \Sigma_c &=(A_c,B_c,C_c,D;\cX^c,\cU,\cY;\kappa)\label{res-contro} \\
     \Sigma_s &=(A_s,B_s,C_s,D;\cX^s,\cU,\cY;\kappa)\label{res-simp}\\
       \Sigma' &=(A',B',C',D; \overline{P_{\cX^o}\cX^c},\cU,\cY;\kappa)\label{res-fmini}
  \\  \Sigma'' &=(A'',B'',C'',D; \overline{P_{\cX^c}\cX^o},\cU,\cY;\kappa)\label{res-smini}
  \end{align} of $\Sigma$ are passive, and $\Sigma_o$ is  observable, $\Sigma_c$ is controllable, $\Sigma_s$ is  simple, and  $\Sigma'$ and  $\Sigma''$ are minimal. For any $n \in \dN_0$ and any $z$ in a sufficiently small symmetric neighbourhood of the origin, it holds
  \begin{align}
    A^nB&=A_c^nB_c=A_s^nB_s, \label{contser}\\
     (I-zA)^{-1}B&=(I-zA_s)^{-1}B_s=(I-zA_c)^{-1}B_c, \label{contser1}\\
     {A^*}^nC^*&={A_o^*}^nC_o^*={A_s^*}^nC_s^*,\label{obsser} \\
     (I-zA^*)^{-1}C^*&=(I-zA^*_s)^{-1}C^*_s=(I-zA^*_c)^{-1}C^*_c. \label{obsser1}
  \end{align}
  Moreover, if $\Sigma$ is co-isometric (isometric), then so are $\Sigma_o$ and $\Sigma_s $ ($\Sigma_c$ and $\Sigma_s$).
\end{lemma}
\begin{proof} Since  $(\cX^o)^\perp$, $(\cX^c)^\perp$ and $(\cX^s)^\perp$ are Hilbert spaces,  the spaces $\cX^o$, $\cX^c$ and $\cX^s$ are regular subspaces  with the negative index $\kappa.$
  It follows from the identities \eqref{cont1}--\eqref{simp1} that  \begin{equation}\label{perties}   \begin{cases}
          (\cX^o)^\perp,(\cX^s)^\perp \,   \mbox{are $A$-invariant, }  \\
           (\cX^c)^\perp,(\cX^s)^\perp \,   \mbox{are $A^*$-invariant, }  \\
          \mathrm{ran}\,C^* \subset {\cX}^o\subset{\cX}^s, \\
           \mathrm{ran}\,B \subset {\cX}^c\subset{\cX}^s,
        \end{cases},    \end{equation} and the representations \eqref{rep-obse}--\eqref{rep-simp} follow. That is,    $\Sigma_o,\Sigma_c$ and $\Sigma_s$ are restrictions of the passive system $\Sigma,$ ans therefore they are passive.

   Let $T_{\Sigma_k}$ be the system operator of $\Sigma_k$ where $k=o,c,s$, and let $\hat{x} \in \cX^k \oplus \cU $ and  $\breve{{x}} \in \cX^k \oplus \cY.$ 
    Calculation show that   \begin{align*}T_{\Sigma_k}\hat{x}&=T\hat{x}, \qquad    k=c,s,\\T_{\Sigma_k}^*\breve{x}&=T^*\breve{x}, \qquad  \! \! k=o,s. \\ \end{align*} It follows from the equations above that  
     if $\Sigma$ is co-isometric (isometric), then so are $\Sigma_o$ and $\Sigma_s $ ($\Sigma_c$ and $\Sigma_s$).

      Suppose  $x \in{ \cX}^o $ such that $C_{o}A_{o}^nx=0$ for every $n=0,1,2,\ldots$. Then
   $$ CA^nx= \begin{pmatrix}0 & C_{o} \end{pmatrix} \begin{pmatrix}
                   A_{1} & A_{2} \\
                    0 & A_{o}\end{pmatrix}^n \begin{pmatrix}
                                                0 \\
                                                x
                                              \end{pmatrix}        =    C_{0}A_{0}^nx=0,       $$ and the identity \eqref{obs1}   implies that $x \in { \cX}^o \cap  (\cX^o)^\perp =\{0\}. $ Thus $x=0,$  and it can be deduced that $ \Sigma_{o}$ is
observable. Similar arguments show that $ \Sigma_{c}$ is controllable and $ \Sigma_{s}$ is simple, the details will be omitted.

Let $u \in \cU,$ and $n \in \dN_0.$ Then, by \eqref{rep-contro} and \eqref{rep-simp},
\begin{align*}
  A^nBu&= \begin{pmatrix}
          A_3 & 0 \\
          A_4 & A_c
        \end{pmatrix}^n   \begin{pmatrix}
          0 \\ B_c
        \end{pmatrix}= \begin{pmatrix}
          0 \\A_c^n B_cu
        \end{pmatrix}= A_c^n B_cu \\ A^nBu&=
         \begin{pmatrix}
          A_5 & 0 \\
          0 & A_s
        \end{pmatrix}^n   \begin{pmatrix}
          0 \\ B_s
        \end{pmatrix} =\begin{pmatrix}
          0 \\A_s^n B_su
        \end{pmatrix}=A_s^n B_su,
\end{align*} and \eqref{contser} holds. By Neumann series, $$(I-zA)^{-1}B=\sum_{n=0}^\infty z^nA^nB$$ holds for all $z$ in a sufficiently small symmetric neighbourhood of the origin, and \eqref{contser1} follows now from \eqref{contser}. The equalities \eqref{obsser} and \eqref{obsser1} can be deduced similarly.

Since the orthocomplements $(\cX^o)^\perp$ and $(\cX^c)^\perp$  are Hilbert subspaces, it follows from  \cite[Lemma 3.1]{Saprikin1} that  $\overline{P_{\cX^o}\cX^c}$ and $ \overline{P_{\cX^c}\cX^o}$ are regular subspaces, and $$\cX^o \cap (P_{\cX^o}\cX^c)^\perp=\cX^o \cap (\cX^c)^\perp, \qquad \cX^c \cap (P_{\cX^c}\cX^o)^\perp=\cX^c \cap (\cX^o)^\perp.$$
Since  $ ({\cX^o})^\perp \subset(P_{\cX^o}\cX^c)^\perp , ({\cX^c})^\perp \subset(P_{\cX^c}\cX^o)^\perp $ and all the spaces are regular, simple calculations show that  \begin{align*} (P_{\cX^o}\cX^c)^\perp &=  (\cX^o)^\perp  \oplus (\cX^o \cap (P_{\cX^o}\cX^c)^\perp ), \\   (P_{\cX^c}\cX^o)^\perp &=  (\cX^c)^\perp  \oplus (\cX^c \cap (P_{\cX^c}\cX^o)^\perp ).   \end{align*}
 Therefore,
  \begin{align*}  \cX&=P_{\cX^o}\cX^c   \oplus (P_{\cX^o}\cX^c)^\perp = (\cX^o)^\perp \oplus \overline{P_{\cX^o}\cX^c} \oplus (\cX^o \cap (P_{\cX^o}\cX^c)^\perp )\\
  &=  (\cX^o)^\perp \oplus \overline{P_{\cX^o}\cX^c} \oplus (\cX^o \cap (\cX^c)^\perp) ,  \end{align*}  and similarly,
  $ \cX=(\cX^c \cap (\cX^o)^\perp)    \oplus \overline{P_{\cX^c}\cX^o} \oplus  (\cX^c)^\perp. $
 Since $(\cX^o \cap (\cX^c)^\perp$  and $\cX^c \cap (\cX^o)^\perp$ are also Hilbert spaces, the spaces $ \overline{P_{\cX^o}\cX^c}$ and $ \overline{P_{\cX^c}\cX^o}$ are Pontryagin spaces with the negative index $\kappa.$
 By considering the properties in \eqref{perties},   the representations \eqref{rep-fmini} and \eqref{rep-smini} follow now easily. That is, $\Sigma'$ and $\Sigma''$ are restrictions of $\Sigma$, and therefore passive.

 Denote $\cX':= \overline{P_{\cX^o}\cX^c} .$ Represent the system operator $T$ of $\Sigma$ as in \eqref{rep-fmini}. Then
 \begin{align*}
  P_{\cX'}A^nB = P_{\cX'}   \begin{pmatrix}A_{11}' & A_{12}' & A_{13}'\\
                                                                     0 & A' & A_{23}'     \\
                                                                     0 & 0 & A_{33}' \end{pmatrix}^n \begin{pmatrix}   B_1' \\                     B'   \\ 0   \end{pmatrix} = \begin{pmatrix}   0 \\       A'^n               B'   \\ 0   \end{pmatrix} = A'^n               B',
 \end{align*}
 and similarly $A'^{*n}C^{'*}  =P_{\cX'}A^{*n}C^*.$ Therefore,
\begin{align*}
{\cX'}^{c} &=\cspan\{\ran A'^{n}B':\,n=0,1,\ldots\} =\cspan\{\ran P_{\cX'}A^nB:\,n=0,1,\ldots\} \\
&=\overline{P_{\cX'}\cspan\{\ran A^nB:\,n=0,1,\ldots\}}=\overline{P_{\cX'}\cX^c}=\overline{P_{\cX'}P_{\cX^o}\cX^c}=\overline{P_{\cX'}\cX'}\\&=\cX',\end{align*}
 and similarly ${\cX'}^{o}=P_{\cX'}\cX^o=\cX'$, which implies that $\Sigma'$ is minimal. A similar argument shows that $\Sigma''$ is minimal, and the proof is complete.
\end{proof}
Note that in particular, Lemma \ref{repres} implies the existence of a minimal passive realization of $\theta \in \SK.$

\begin{definition}\label{restricts}
The restrictions $\Sigma_o,\Sigma_c,\Sigma_s,\Sigma',$ and $\Sigma''$ in Lemma  \ref{repres}  are called, respectively, the observable,  the controllable, the simple (or proper), the \textbf{first minimal} and the \textbf{second minimal} restrictions of $\Sigma.$ \end{definition}

The first minimal and the second minimal restrictions will be considered later in Sections \ref{sec-opmin} and \ref{sec7}.


Two  realizations $$\Sigma_1=(A_1,B_1,C_1,D_1;\cX_1,\cU,\cY;\kappa_1), \qquad\Sigma_2=(A_2,B_2,C_2,D_2;\cX_2,\cU,\cY;\kappa_2)$$
 of the same function $\theta \in \SK$ are  called \textbf{unitarily similar} if $D_1=D_2$ and there exists a unitary operator $U: \cX_1 \to \cX_2 $
  such that  \begin{equation}\label{unisim}  A_1=U^{-1}A_2U, \quad B_1=U^{-1}B_2, \quad C_1=C_2U.     \end{equation} In that case, it easily follows that $\kappa_1=\kappa_2.$ Unitary similarity preserves dynamical properties of the system and also the spectral properties of the main operator.
  If two realizations of $\theta \in \SK$ both have the same property {\rm(i), (ii)} or { \rm(iii)} of Theorem \ref{realz},
   then they are unitarily similar \cite[Theorem 2.1.3]{ADRS}.

   The realizations $\Sigma_1$ and $\Sigma_2$ above are said to be \textbf{weakly similar} if  $D_1=D_2$
  and there exists an injective closed densely defined possible unbounded linear operator $Z: \cX_1 \to \cX_2$ with the dense range such that
 \begin{equation}\label{weaksim}
ZA_1x=A_2Zx, \quad C_1x=C_2Zx, \quad  x\in\cD(Z),\quad  \mbox{and} \quad
ZB_1=B_2,      \end{equation} where $\cD(Z)$ is the domain of $Z.$
In Hilbert state space case,
    a result of Helton \cite{Helton} and Arov \cite{A} states that two minimal passive realizations of
    $\theta \in \So$ are weakly similar. However, weak similarity preserves neither
     dynamical properties of the system nor the spectral properties of its main operator.

       Helton's and Arov's statement holds also in case where all the spaces are indefinite. This result is stated for reference purposes.
       Similar argument as Hilbert space case can be applied, 
        definiteness of the inner product play no role. For a proof of special cases, see  \cite[Theorem 7.1.3]{BGKR}, \cite[p. 702]{Staffans} and \cite[Theorem 2.5]{Lassi}. Note that the  realizations are not assumed to be $\kappa$-admissible or passive.
\begin{proposition}\label{weaksim1}
  Two minimal  realizations of $\theta \in \SK$ are weakly similar.
\end{proposition}

          \section{Optimal minimal  systems}\label{sec-opmin}
          For $\kappa$-admissible realizations of $\theta \in \SK$, where $\cU$ and $\cY$ are Pontryagin spaces with the same negative index, one can form the similar theory of optimal minimal passive systems as represented  in the standard Hilbert space case in \cite{ArKaaP} and the Pontryagin state space case in \cite{Saprikin1}.
          Techniques, definitions and notations to be used here are similar to what appears in those papers.

 Denote  $E_{\cX}\left( x \right) =\langle x,x \rangle_{\cX}$ for a vector $x$ in an inner product space $\cX.$ The same notation has been used in the proof of Proposition \ref{consdil}. Following \cite{ArKaaP,SapAr,Saprikin1},
 a passive realization $\Sigma=(A,B,C,D;\cX,\cU,\cY;\kappa)$ of $\theta \in \SK$ is called \textbf{optimal} if for any passive realization $\Sigma'=(A',B',C',D';\cX',\cU,\cY;\kappa)$ of $\theta,$ the inequality
\begin{equation}\label{optimal}
  E_{\cX}\left( \sum_{k=0}^{n}A^kBu_k \right) \leq    E_{\cX'}\left( \sum_{k=0}^{n}A'^kB'u_k  \right), \qquad n \in \dN_0, \quad u_k \in \cU,
\end{equation}
holds. On the other hand, the  system $\Sigma$  is called \textbf{*-optimal} if it is observable and
\begin{equation}\label{*-optimal}
  E_{\cX}\left( \sum_{k=0}^{n}A^kBu_k \right) \geq    E_{\cX'}\left( \sum_{k=0}^{n}A'^kB'u_k  \right), \qquad n \in \dN_0, \quad u_k \in \cU,
\end{equation} holds for every observable passive realization $\Sigma'$ of $\theta.$ The requirement for observability must be included for avoiding trivialities, since otherwise every isometric realization of $\theta$ would be $^*$-optimal; see Lemma \ref{canassume} below and \cite[Proposition 3.5 and example on page 144]{ArKaaP}.

In the definition of optimality, the requirement that the considered realizations are $\kappa$-admissible is essential, as the example below shows.
\begin{example}
  Let $$\sys, \quad\Sigma'=(A',B',C',D';\cX',\cU,\cY;\kappa'),$$ where $\kappa<\kappa',$ be passive realization of $\theta \in \SK.$ Suppose that \eqref{optimal} holds. By Lemma \ref{repres}, if \eqref{optimal} holds for $\Sigma$, it holds also for the controllable restriction  $\Sigma_c=(A_c,B_c,C_c,D';\cX^c,\cU,\cY;\kappa)$  of $\Sigma.$ For any vector $x$ of the form $$x=\sum_{n=0}^{M}A_c^nB_cu_n, \qquad \{u_n\}\subset \cU,  \qquad M \in \dN_0, $$ define
  $$   Rx= \sum_{n=0}^{M}A'^nB'u_n.  $$ It is easy to deduce that $R$ is a linear relation. Moreover, since $\Sigma_c$ is controllable by Lemma \ref{repres}, $R$ is densely defined. Since \eqref{optimal} holds, $R$ is contractive. It follows now from \cite[Theorem 1.4.2]{ADRS} that $R$ can be extended to be everywhere defined contractive linear operator. Since $\mathrm{ind}_- \cX^c=\kappa < \kappa'=\mathrm{ind}_- \cX',$ it follows from \cite[Theorem 2.4]{rovdrit} that linear operator from $\cX^c$ to $\cX'$ cannot be contractive, and hence \eqref{optimal} cannot hold.
\end{example}


It will be shown in Theorem \ref{optimals} below that an optimal ($^*$-optimal) minimal realization exists, and it can be constructed by taking the first (second) minimal restriction, introduced in Definition \ref{restricts}, of simple conservative realizations. More lemmas will be needed before that.

                     \begin{lemma} \label{properres}
                         Let $\sys$ is a passive realization of $\theta \in \SK,$ and  let $\Sigma_s=(A_s,B_s,C_s,D;\cX^s,\cU,\cY;\kappa)$  be the restriction of $\Sigma$ to the simple subspace. Then, the first (second) minimal restrictions of $\Sigma$ and $\Sigma_s$ coinside.
                     \end{lemma} \begin{proof} Only the proof of the statement conserning about the second minimal restrictions is provided, since the other case is similar. To make notation less cumbersome, write $\cX^s=\cX_p,$ where $p$ refers to proper part. By Lemma \ref{repres}, the equalities \eqref{contser} and \eqref{obsser} hold, and it easily follows that
                          \begin{align*} \cX^o &= \cX_p^o,&   \cX^c &= \cX_p^c  \\
                          (\cX^o)^\perp &= (\cX^s)^\perp \oplus (\cX_p^o)^\perp,& (\cX^c)^\perp  &=  (\cX^s)^\perp \oplus(\cX_p^c)^\perp,   \end{align*}    where orthogonal complements  $ (\cX_p^o)^\perp$ and  $(\cX_p^c)^\perp$ are taken with respect to the space $\cX_p.$ Therefore,
                          $$ P_{\cX^c}\cX^o =  P_{ \cX_p^c}\cX_p^o      \subset \cX^s=\cX_p,      $$ and consequently,
                          \begin{align*}P_{ P_{\cX_p^c}\cX_p^o }A_p\uphar_{P_{\cX_p^c}\cX_p^o} &=  P_{ P_{\cX^c}\cX^o }A\uphar_{\cX^s}  \uphar_{P_{\cX^c}\cX^o}   = P_{ P_{\cX^c}\cX^o }A  \uphar_{P_{\cX^c}\cX^o},  \\
                                                                                                P_{ P_{\cX_p^c}\cX_p^o }B_p  &= P_{ P_{\cX^c}\cX^o }B, \\C_p\uphar_{{ P_{\cX_p^c}\cX_p^o }  } &=C  \uphar_{ { P_{\cX^c}\cX^o }},\end{align*} which shows that the second minimal restrictions of $\Sigma$ and $\Sigma_s$ co-inside.
                          \end{proof}
                          To prove the ($^*$-)optimality of a system, the following lemma is helpful. 
                          \begin{lemma}\label{canassume}
                             Let \begin{align*} {\Sigma}=({A},{B},{C},D;{\cX},\cU,\cY,&\kappa), \qquad   \widehat{\Sigma}=(\widehat{A},\widehat{B},\widehat{C},D;\widehat{\cX},\cU,\cY,\kappa),  \\ \Sigma'=(&A',B',C',D;\cX',\cU,\cY;\kappa), \end{align*} be realizations of $\theta \in \SK$ such that $\Sigma$ is  passive, $\widehat{\Sigma}$ is a passive dilation of $\Sigma$ and $\Sigma'$ is the first minimal restriction of $\widehat{\Sigma}.$
                            Then
                            \begin{align}\label{dil-re-op}
  E_{\cX'}\left(    \sum_{k=0}^{n}A'^kB'u_k    \right)                        \leq  E_{\cX}\left( \sum_{k=0}^{n}A^kBu_k     \right) ,\qquad n \in \dN_0, \quad u_k \in \cU.
                            \end{align}           Moreover, for any isometric realization  $    {\Sigma_1}=({A_1},{B_1},{C_1},D;{\cX_1},\cU,\cY,\kappa)$ of $\theta,$ it holds
                             \begin{align}\label{iso-star}
      E_{\cX}\left( \sum_{k=0}^{n}A^kBu_k     \right)  \leq    E_{\cX_1}\left( \sum_{k=0}^{n}A_1^kB_1u_k \right) ,\qquad n \in \dN_0, \quad u_k \in \cU.
                            \end{align}
                          \end{lemma}
Note that Proposition \ref{consdil} quarantees the existence of a passive dilation $\widehat{{\Sigma}}$ of $\Sigma$ with the properties discribed above.
                          \begin{proof}
                            Since $\widehat{\Sigma}$ is a dilation of $\Sigma,$ 
                             the system operator   $T_{\widehat{\Sigma}}$ has a representation
                            \begin{equation}\label{dils} T_{\widehat{\Sigma}}\! =\!\begin{pmatrix}
                                                       \widehat{A} & \widehat{B} \\
                                                       \widehat{C} & D
                                                     \end{pmatrix}\!=\!\begin{pmatrix}
                                                                     \begin{pmatrix}A_{11} & A_{12} & A_{13}\\
                                                                     0 & A & A_{23}     \\
                                                                     0 & 0 & A_{33} \end{pmatrix} & \begin{pmatrix}   B_1 \\                     B   \\ 0   \end{pmatrix} \\
                                                                    \begin{pmatrix} 0 & C_1 & C \end{pmatrix} & D
                                                                   \end{pmatrix} \!:\!\begin{pmatrix}
                                                                                    \begin{pmatrix}
                                                                                      \cD \\
                                                                                      \cX \\
                                                                                      \cD_*
                                                                                    \end{pmatrix} \\
                                                                                    \cU
                                                                                  \end{pmatrix} \!\to\! \begin{pmatrix}
                                                                                    \begin{pmatrix}
                                                                                     \cD \\
                                                                                      \cX \\
                                                                                      \cD_*
                                                                                    \end{pmatrix} \\
                                                                                    \cY
                                                                                  \end{pmatrix},   \end{equation} where $\cD$ and $\cD^*$ are Hilbert spaces. On the other hand,  by Lemma \ref{repres},
 $\widehat{\Sigma}$ can also be represented as
  $$ T_{\widehat{\Sigma}} 
                                                     =\begin{pmatrix}
                                                                     \begin{pmatrix}A_{11}' & A_{12}' & A_{13}'\\
                                                                     0 & A' & A_{23}'     \\
                                                                     0 & 0 & A_{33}' \end{pmatrix} & \begin{pmatrix}   B_1' \\                     B'   \\ 0   \end{pmatrix} \\
                                                                    \begin{pmatrix} 0 &C' &   C_1' \end{pmatrix} & D
                                                                   \end{pmatrix} :\begin{pmatrix}
                                                                                    \begin{pmatrix}
                                                                                      \cX_1 \\
                                                                                      \cX' \\
                                                                                      \cX_3
                                                                                    \end{pmatrix} \\
                                                                                    \cU
                                                                                  \end{pmatrix} \to \begin{pmatrix}
                                                                                    \begin{pmatrix}
                                                                                     \cX_1 \\
                                                                                      \cX' \\
                                                                                     \cX_3
                                                                                    \end{pmatrix} \\
                                                                                    \cY
                                                                                  \end{pmatrix},   $$ where $\cX_1=(\widehat{\cX}^o)^\perp, \cX'= \overline{P_{\widehat{\cX}^o}\widehat{\cX}^c}$ and $\cX_3=\widehat{\cX}^o \cap(\widehat{\cX}^c)^\perp.$ The spaces $\cX_1$ and $\cX_3$ are Hilbert spaces, and $\cX'$ is a Pontryagin space with the negative index $\kappa.$ Let $n \in \dN_0$ and $\{u_k \}_{k=0}^n \subset \cU.$ Since $\cX_3 \subset (\widehat{\cX}^c)^\perp,$ it holds
                                                                                  \begin{align}\label{eqrep1}
                                                                                   E_{\cX'}\left(    \sum_{k=0}^{n}A'^kB'u_k    \right) &=    E_{\widehat{\cX}}\left(P_{\cX'}    \sum_{k=0}^{n}\widehat{A}^k\widehat{B}u_k    \right) \notag \\
                                                                                   &= E_{\widehat{\cX}}\left(   \sum_{k=0}^{n}\widehat{A}^k\widehat{B}u_k    \right) - E_{\widehat{\cX}}\left(P_{\cX_1}    \sum_{k=0}^{n}\widehat{A}^k\widehat{B}u_k    \right)\notag \\& \quad
                                                                                    - E_{\widehat{\cX}}\left(P_{\cX_3}    \sum_{k=0}^{n}\widehat{A}^k\widehat{B}u_k    \right)\notag \\
                                                                                    &= E_{\widehat{\cX}}\left(   \sum_{k=0}^{n}\widehat{A}^k\widehat{B}u_k    \right) - E_{\widehat{\cX}}\left(P_{\cX_1}    \sum_{k=0}^{n}\widehat{A}^k\widehat{B}u_k    \right).
                                                                                  \end{align}
                                                                                  With $\cD$ and $\cD^*$ as in \eqref{dils}, the identities in \eqref{dilation} hold.
                                                                                  Therefore, it follows from the identities \eqref{cont1}  and \eqref{obs1} that   $\cD_* \subset (\widehat{\cX}^c)^\perp $ and $\cD \subset
                                                                                   (\widehat{\cX}^o)^\perp = \cX_1. $  A similar calculation as above yields then
                                                                                  \begin{align}\label{eqrep2}
                                                                                     E_{\cX}\left(    \sum_{k=0}^{n}A^kBu_k    \right)=E_{\widehat{\cX}}\left(   \sum_{k=0}^{n}\widehat{A}^k\widehat{B}u_k    \right) - E_{\widehat{\cX}}\left(P_{\cD}    \sum_{k=0}^{n}\widehat{A}^k\widehat{B}u_k    \right).
                                                                                  \end{align} 
                                                                                   Since $\cX_1$ is  a Hilbert space, 
                                                                                    the inclusion $\cD \subset  
                                                                                    \cX_1 $ implies
                                                                                  $$E_{\widehat{\cX}}\left(P_{\cD}    \sum_{k=0}^{n}\widehat{A}^k\widehat{B}u_k    \right)\leq         E_{\widehat{\cX}}\left(P_{\cX_1}    \sum_{k=0}^{n}\widehat{A}^k\widehat{B}u_k    \right). $$ It follows now from the equations \eqref{eqrep1} and \eqref{eqrep2} that
                                                                              $$E_{\cX'}\left(    \sum_{k=0}^{n}A'^kB'u_k    \right)  \leq      E_{\cX}\left(    \sum_{k=0}^{n}A^kBu_k    \right),   $$ and the inequality \eqref{dil-re-op} is proved.

                                                                              Assume  that $\widehat{{\Sigma}}$ is isometric. Since $\cD$ is a Hilbert space,  it follows from \eqref{eqrep2} that
                                                                                \begin{align} \label{iso-star1}
                                                                                     E_{\cX}\left(    \sum_{k=0}^{n}A^kBu_k    \right)\leq E_{\widehat{\cX}}\left(   \sum_{k=0}^{n}\widehat{A}^k\widehat{B}u_k    \right) .
                                                                                  \end{align}
                                                                                  By Lemma \ref{repres}, the system operator of $\widehat{\Sigma}$ can be represented as
                                                                                   $$ T_{\widehat{\Sigma}} = \begin{pmatrix}
        \begin{pmatrix}
          A_3 & 0 \\
          A_4 & A_c
        \end{pmatrix} &  \begin{pmatrix}
          0 \\ B_c
        \end{pmatrix}\\
        \begin{pmatrix}
          C_1 & C_c
        \end{pmatrix} & D
      \end{pmatrix}: \begin{pmatrix}
                                                                                                                             \begin{pmatrix}
                                                                                                                               (\widehat{\cX}^c)^\perp \\
                                                                                                                               \widehat{\cX}^c
                                                                                                                             \end{pmatrix} \\
                                                                                                                             \cU
                                                                                                                           \end{pmatrix} \to \begin{pmatrix}
                                                                                                                             \begin{pmatrix}
                                                                                                                               (\widehat{\cX}^c)^\perp \\
                                                                                                                               \widehat{\cX}^c
                                                                                                                             \end{pmatrix} \\
                                                                                                                             \cY
                                                                                                                           \end{pmatrix},     $$ where the restriction ${\Sigma_c}=({A_c},{B_c},{C_c},D;{\widehat{\cX}^c},\cU,\cY,\kappa)$ is controllable isometric, and for every $n=0,1,2,\ldots,$ it holds
                                                                                                                           $  \widehat{A}^n\widehat{B}=A_c^nB_c. $ Therefore
                                                                                                                           \begin{align}\label{iso-star2}
                                                                                      E_{\widehat{\cX}}\left(   \sum_{k=0}^{n}\widehat{A}^k\widehat{B}u_k    \right) = E_{\widehat{\cX}^c}\left(   \sum_{k=0}^{n}{A_c}^k{B_c}u_k    \right) .
                                                                                  \end{align} Similar argument show that if  $    {\Sigma_1^c}=({A_1^c},{B_1^c},{C_1^c},D;{\cX_1^c},\cU,\cY,\kappa)$ is the restriction of the isometric system  $    {\Sigma_1}=({A_1},{B_1},{C_1},D;{\cX_1},\cU,\cY,\kappa)$ to the controllable subspace $\cX_1^c,$   then $\Sigma_1^c$ is controllable isometric and it holds
                                                                                  \begin{align}\label{iso-star3}
                                                                                      E_{{\cX_1}}\left(   \sum_{k=0}^{n}{A_1}^k{B_1}u_k    \right) = E_{{\cX_1^c}}\left(   \sum_{k=0}^{n}{A_1^c}^k{B_1^c}u_k    \right).
                                                                                               \end{align}                            But ${\Sigma_0}$ and ${\Sigma_1^c}$ are unitarily similar, and therefore
                                                                                                 \begin{align}\label{iso-star4}
                                                                                      E_{\widehat{\cX}^c}\left(   \sum_{k=0}^{n}{A_0}^k{B_0}u_k    \right)  = E_{{\cX_1^c}}\left(   \sum_{k=0}^{n}{A_1^c}^k{B_1^c}u_k    \right).
                                                                                               \end{align}  By combining \eqref{iso-star1}--\eqref{iso-star4}, the inequality \eqref{iso-star} follows. \end{proof}

\begin{remark}
  \label{iso-star-on}It follows from the inequality \eqref{iso-star} of    Lemma \ref{canassume}  that if there exists an observable isometric realization of $\theta \in \SK,$ then it is $^*$-optimal.
\end{remark}

                In the standard Hilbert space case,  results of Arov \cite{Arov} show that there exist   
                optimal minimal realizations of a Schur function. 
                The construction was based on the existence of the defect functions, see Section \ref{sec7}.
                 Arov et. all provided new geometric proofs of these results in \cite{ArKaaP}. Saprikin used those new proofs and generalized Arov's results to Pontryagin state space case in \cite{Saprikin1}. It will be proved next that Arov's results holds in the case when all spaces are Pontryagin spaces. The geometric proofs in \cite{ArKaaP} can still be applied in the precent setting with few appropriate changes.


\begin{theorem}\label{optimals}
Let $\theta \in \SK$, where $\cU$ and $\cY$ are Pontryagin spaces with the same negative index. Then:
\begin{itemize}
  \item[{\rm(i)}] The first minimal restriction of a simple conservative realization of $\theta$ is optimal  minimal;
\item[{\rm(ii)}]  The minimal passive system $\Sigma^*$ is optimal if and only if the dual system $\Sigma$ is *-optimal minimal;
\item[{\rm(iii)}]The second minimal restriction of a simple conservative realization of $\theta$ is *-optimal  minimal;
\item[{\rm(iv)}] Optimal (*-optimal) minimal   systems are unique up to unitary similarity, and every optimal (*-optimal) minimal  realization of $\theta$ is the first minimal restriction (second minimal restriction) of some simple conservative realization of $\theta.$
 \end{itemize}
\end{theorem}
\begin{proof}
  {\rm(i)} Let $\Sigma'=(A',B',C',D;\cX',\cU,\cY;\kappa)$ be the first minimal restriction of a simple conservative realization $\widehat{\Sigma}'=(\widehat{A}',\widehat{B}',\widehat{C}',D;\widehat{\cX}',\cU,\cY;\kappa)$ of $\theta \in \SK.$ Let $\sys$ be the first minimal restriction of some conservative realization of $\theta$ such that its state space has negative index $\kappa.$ To prove that $\Sigma'$ is optimal, Lemma \ref{canassume} shows that it is enough to prove
 \begin{equation}\label{enog} E_{\cX'}\left( \sum_{k=0}^{n}A'^kB'u_k  \right)\leq   E_{\cX}\left( \sum_{k=0}^{n}A^kBu_k \right)   , \qquad n \in \dN_0, \quad u_k \in \cU. \end{equation}
By  Lemma \ref{properres}, it can be assumed that $\Sigma$ is the first minimal restriction of some  simple conservative realization $\widehat{\Sigma}=(\widehat{A},\widehat{B},\widehat{C},D;\widehat{\cX},\cU,\cY;\kappa)$ of $\theta.$ Since  $\widehat{\Sigma}$ and $\widehat{\Sigma}'$ are both simple conservative, they are unitarily similar, so there exists a unitary operator $U: \widehat{\cX} \to \widehat{\cX}'$  such that
 $$  \widehat{A}=U^{-1}\widehat{A}'U, \quad \widehat{B}=U^{-1}\widehat{B}', \quad \widehat{C}=\widehat{C}'U.   $$ Easy calculations shows that $\widehat{\cX}^{'o} =U\widehat{\cX}^{o},$ $\widehat{\cX}^{'c} =U\widehat{\cX}^{c},$ $(\widehat{\cX}^{'o})^\perp =U(\widehat{\cX}^{o})^\perp,$  $(\widehat{\cX}^{'c})^\perp =U(\widehat{\cX}^{c})^\perp$ and  $P_{\widehat{\cX}^{'o}}\widehat{\cX}^{'c} = U  P_{\widehat{\cX}^{o}}\widehat{\cX}^{c}.$
In particular, $$P_{\cX}=P_{P_{\widehat{\cX}^{o}}\widehat{\cX}^{c}}=   U^{-1} P_{P_{\widehat{\cX}^{'o}}\widehat{\cX}^{'c}}U =U^{-1} P_{\cX'} U,$$
  which implies
 \begin{align*}
 A&= P_{\cX} \widehat{A} \uphar_{\cX} = U^{-1} P_{\cX'} \widehat{A}'U \uphar_{\cX} =  (U\uphar_{\cX})^{-1} P_{\cX'} \widehat{A}' \uphar_{\cX'} U\uphar_{\cX}= (U\uphar_{\cX})^{-1}  {A}'  U\uphar_{\cX}\\
 B&=(U\uphar_{\cX})^{-1}{B}', \qquad C= C'U\uphar_{\cX}.
 \end{align*} It follows that $\Sigma$ and $\Sigma'$ are unitarily similar and the corresponding   unitary operator is $U_0=U\uphar_{\cX}.$  Then
 \begin{align*} E_{\cX}\left( \sum_{k=0}^{n}A^kBu_k  \right) = E_{\cX}\left(U_0^{-1} \sum_{k=0}^{n}A'^kB'u_k  \right)= E_{\cX'}\left( \sum_{k=0}^{n}A'^kB'u_k  \right),
 \end{align*} Then \eqref{enog} holds, and $\Sigma'$ is an optimal minimal system.

 {\rm(ii)}  Let    $\Sigma^*=(A^*,C^*,B^*,D^*;\cX,\cY,\cU;\kappa)$ be an optimal minimal passive realization of ${\theta}^\# \in \mathbf{S}_{\kappa}(\cY,\cU).$  Then $\sys$ is a minimal passive realization of $\theta \in \SK.$  Consider  an arbitrary observable passive realization  ${\Sigma'}=({A'},{B'},{C'},{D};{\cX'},\cU,\cY;\kappa)$ of $\theta \in \SK.$  Then ${\Sigma}^{'*} = (A^{'*},C^{'*},B^{'*},D^*;\cX',\cY,\cU;\kappa)  $ is a controllable passive realization of ${\theta^\#}.$  For a vector of the form
              $$x'= \sum_{k=0}^{n}{{A^{'*}}}^k{C}^{'*}y_k, \qquad n \in \dN_0, \quad y_k \in \cY,$$
               define $$Sx'= \sum_{k=0}^{n}{(A^*)}^kC^*y_k.$$ Since $\Sigma^{'*}$ is controllable and $\Sigma^{*}$  is optimal, the domain of $S$ is dense, and it holds
              $$E_{\cX}(Sx)= E_{\cX}\left( \sum_{k=0}^{n}{(A^*)}^kC^*y_k \right) \leq    E_{\cX'}\left( \sum_{k=0}^{n}{{A^{'*}}}^k{C}^{'*}y_k  \right)=E_{\cX'}(x).      $$
               That is, $S$ is a contractive linear relation with the dense domain. Then  \cite[Theorem 1.4.4]{ADRS} shows that the closure of $S$, which is still denoted as $S$, is contractive everywhere defined linear operator from $\cX' \to \cX.$
                 Since $\cX'$ and $\cX$ are Pontryagin spaces with the same negative index, $S^*: \cX \to \cX',$ is contractive as well. The transfer functions of the $\Sigma$ and $\Sigma'$ coincide, and therefore $CA^mB=C'{A'}^kB'$ for every $m \in \dN_0.$ By definition,   $S({A^{'*}})^mC^{'*}=({A^{*}})^mC^{*},$  or what is the same thing, $C'{A'}^mS^*=CA^m,$ for every $m \in \dN_0.$ Then also  $$C'{A'}^{m+k}B'=C{A}^{m}{A}^{k}B=C'{A'}^mS^*A^{k}B \quad \text{for}\quad m,k\geq0.$$  This implies ${A'}^{k}B'=S^*A^{k}B$  since the system $\Sigma'$  is observable. It  follows now that
              $$  S^*\left(
              \sum_{k=0}^{n}{A}^kBu_k\right)
               =   \sum_{k=0}^{n}{A'}^kB'u_k.     $$ Therefore,
              $$       E_{\cX'}\left(  \sum_{k=0}^{n}{A'}^kB'u_k \right)  =   E_{\cX'}\left(   S^* \left(\sum_{k=0}^{n}{A}^kBu_k\right) \right) \leq   E_{\cX}\left(   \sum_{k=0}^{n}{A}^kBu_k \right),                   $$ since $S^*$ is contractive. This proves that $\Sigma$ is  $^*$-optimal.

              Suppose then that $\sys$ is minimal passive   $^*$-optimal realization of $\theta \in  \SK.$ Then $\Sigma^*$ is a minimal passive realization of ${\theta}^\# \in \mathrm{\mathbf{S}}_{\kappa}(\cY,\cU).$ To prove the optimality of $\Sigma^*$, it suffices to consider all the minimal passive realizations of ${\theta}^\#$; see Lemma \ref{canassume}. Let $\Sigma^{'*}=(A^{'*},C^{'*},B^{'*},D^{*};\cX',\cY,\cU;\kappa)$ be a minimal passive realization of ${\theta^\#}.$ Then $\Sigma'$ is a minimal passive realization of $\theta. $  Since $\Sigma$ is $^*$-optimal, the inequality
              $$  E_{\cX}\left( \sum_{k=0}^{n}A^kBu_k \right) \geq    E_{\cX'}\left( \sum_{k=0}^{n}A'^kB'u_k  \right), \qquad n \in \dN_0, \quad u_k \in \cU,    $$ holds. Define $Kx=\sum_{k=0}^{n}A'^kB'u_k$ for $ x=\sum_{k=0}^{n}A^kBu_k. $ Using similar techniques as above, $K$ can be extended to be a contractive operator  from $\cX \to \cX'$ such that  $$K^*(A^{'*})^kC^{'*}=(A^{*})^kC^{*}.$$
            Since $K^*$ is contractive,
               $$       E_{\cX}\left(  \sum_{k=0}^{n}{A^{*}}^k C^*y_k \right)  =    E_{\cX}\left( K^* \sum_{k=0}^{n}{A^{'*}}^k C^{'*}y_k \right)  \leq   E_{\cX'}\left(  \sum_{k=0}^{n}{A^{'*}}^k C^{'*}y_k \right),                   $$ for $\{ y_k\}\subset \cY.$ This shows that $\Sigma^*$ is optimal.

                 {\rm(iii)} Let $\Sigma$ be a simple conservative realization of $\theta \in \SK.$ Then $\Sigma^*$ is a simple conservative realization of ${\theta^\#},$ and the first minimal restriction ${\Sigma^*}'$ of $\Sigma^*$ is optimal minimal by the part (i). By using the representations \eqref{rep-fmini} and \eqref{rep-smini} from Lemma \ref{repres},   it is easy to deduce that the dual system of  ${\Sigma^*}'$ is the second minimal restriction $\Sigma''$ of $\Sigma$, and it follows from the part {\rm(ii)} that $\Sigma''$ is $^*$-optimal.

{\rm(iv)} Only the proofs of the claims considering optimal minimal realizations will be given, since the  claims considering $^*$-optimal minimal realizations can be proved  analogously.
Let $$\Sigma_j=(A_j,B_j,C_j,D;\cX_j,\cU,\cY;\kappa), \quad j=1,2,$$ be optimal minimal realizations of $\theta \in \SK.$  In a sufficiently small neighbourhood of the origin, the transfer functions 
                   $\theta_{\Sigma_1}$ and $\theta_{\Sigma_2}$ of the systems $\Sigma_1$ and $\Sigma_2$
                   have the Neumann series and they coincide,  so $C_1A_1^kB_1=C_2A_2^kB_2$ for $k=0,1,2,\ldots$ Since $\Sigma_1$ is controllable,
                    vectors of the form $$x=\sum_{k=0}^N A_1^k B_1 u_k, \qquad u_k \in \cU,$$ are dense in $\cX_1.$
                     Define \begin{equation}Ux=\sum_{k=0}^N A_2^k B_2 u_k.\label{eks}\end{equation}
                     Because $\Sigma_2$ is controllable as well, vectors of the form $Ux$ are dense in $\cX_2.$
                   Since $\Sigma_1$ and $\Sigma_2$ both are optimal realizations,
                   $E_{\cX_1} (x)  =   E_{\cX_1} (Ux),$ and therefore $U$ is an isometric linear relation with the dense domain and the dense range. It follows now from \cite[1.4.2]{ADRS} that the closure of $U$ is a unitary operator, which  is still denoted as $U.$  Then, trivially $B_1=U^{-1}B_2.$ 
                    For vector $x$  in \eqref{eks}, it holds
                   $$UA_1x= U  \sum_{k=0}^N A_1^{k+1} B_1 u_k  =   \sum_{k=0}^N  A_2^{k+1} B_2 u_k =A_2Ux.$$ It follows that  $UA_1x=A_2Ux$ holds in a dense set, and therefore by continuity, everywhere. Thus $A_1=U^{-1}A_2U.$  Moreover,
                   for $k = 0,1,2,\ldots,$ one concludes
                    $$C_1A_1^kB_1=C_2A_2^kB_2 = C_2UA_1^kB_1.  $$ Since $\mathrm{span}_{k \in \dN_0}A_1^k B_1$ is dense in $\cX_1,$ it must be $C_1=C_2U.$ It has been shown that the unitary operator $U$ has all the properties of \eqref{unisim}, and therefore $\Sigma_1$ and $\Sigma_2$ are unitarily similar.

     Suppose then that $\sys$ is an optimal minimal realization of $\theta.$ Let $\widehat{\Sigma}_0=(\widehat{A}_0,\widehat{B}_0,\widehat{C}_0,D;\widehat{\cX}_0,\cU,\cY;\kappa)$ be some simple conservative realization of $\theta.$
                Lemma \ref{repres} shows that the system operator of $\widehat{\Sigma}$ can be represented as
                $$ T_{\widehat{\Sigma}_0} =\begin{pmatrix}
                                                                     \begin{pmatrix}A_{11}' & A_{12}' & A_{13}'\\
                                                                     0 & A' & A_{23}'     \\
                                                                     0 & 0 & A_{33}' \end{pmatrix} & \begin{pmatrix}   B_1' \\                     B'   \\ 0   \end{pmatrix} \\
                                                                    \begin{pmatrix} 0 &C' &   C_1' \end{pmatrix} & D
                                                                   \end{pmatrix} :\begin{pmatrix}
                                                                                    \begin{pmatrix}
                                                                                     \cX_1 \\
                                                                                      \cX' \\
                                                                                      \cX_2
                                                                                    \end{pmatrix} \\
                                                                                    \cU
                                                                                  \end{pmatrix} \to \begin{pmatrix}
                                                                                      \begin{pmatrix}
                                                                                     \cX_1 \\
                                                                                      \cX' \\
                                                                                      \cX_2
                                                                                    \end{pmatrix} \\
                                                                                    \cY
                                                                                  \end{pmatrix},   $$ where $ \cX_1=(\widehat{\cX}^o)^\perp, \cX'=\overline{P_{\widehat{\cX}^o}\widehat{\cX}^c}$ and $\cX_2= \widehat{\cX}^o \cap (\widehat{\cX}^c)^\perp.$ Then the system $\Sigma'=(A',B',C',D;\cX',\cU,\cY;\kappa)$ is the first minimal restriction of $\widehat{\Sigma},$ and it follows from  part \textrm{(i)} that $\Sigma'$ is optimal minimal, and moreover, as proved above,  unitarily similar with $\Sigma.$
                                                                                Therefore, there exists a unitary operator $U: \cX \to \cX'$ such that
                                                                                  $A =U^{-1}A'U ,B=U^{-1}B' $ and $C=C'U.$ Define
                                                                                   \begin{align*} T_{\widehat{\Sigma}} &=\begin{pmatrix}
                                                       \widehat{A} & \widehat{B} \\
                                                       \widehat{C} & D
                                                     \end{pmatrix}\\&=\begin{pmatrix}
                                                                     \begin{pmatrix}A_{11}' & A_{12}'U & A_{13}'\\
                                                                     0 & A & U^{-1}A_{23}'     \\
                                                                     0 & 0 & A_{33}' \end{pmatrix} & \begin{pmatrix}   B_1' \\                     B   \\ 0   \end{pmatrix} \\
                                                                    \begin{pmatrix} 0 & C & C_1' \end{pmatrix} & D
                                                                   \end{pmatrix} :\begin{pmatrix}
                                                                                    \begin{pmatrix}
                                                                                      \cX_1 \\
                                                                                      \cX \\
                                                                                      \cX_2
                                                                                    \end{pmatrix} \\
                                                                                    \cU
                                                                                  \end{pmatrix} \to \begin{pmatrix}
                                                                                    \begin{pmatrix}
                                                                                      \cX_1 \\
                                                                                      \cX \\
                                                                                      \cX_2
                                                                                    \end{pmatrix} \\
                                                                                    \cY
                                                                                  \end{pmatrix},   \end{align*} and let $\widehat{\Sigma}$ be the system corresponding the  system operator  $T_{\widehat{\Sigma}} $. Easy calculations show that $\widehat{\Sigma}$ and  $\widehat{\Sigma}_0$ are unitarily similar and $$\widehat{U}=\begin{pmatrix}
                                                                                                                                                                            I & 0 &0 \\
                                                                                                                                                                         0 &    U & 0\\
                                                                                                                                                                          0 &  0 & I
                                                                                                                                                                          \end{pmatrix}:\begin{pmatrix}
                                                                                                                                                                           \cX_1 \\
                                                                                      \cX \\
                                                                                      \cX_2
                                                                                                                                                                          \end{pmatrix} \to \begin{pmatrix}
                                                                                                                                                                           \cX_1 \\
                                                                                      \cX' \\
                                                                                      \cX_2
                                                                                                                                                                          \end{pmatrix}
                                                                                      $$ is the corresponding unitary operator. Therefore $\widehat{\Sigma}$ is a simple conservative system. Now $\widehat{U}$ maps $P_{\cX^o}\cX^c$ to $P_{{\cX'}^o}{\cX'}^c,$ and $\widehat{U}{\cX'}=U{\cX'}=\cX.$ It follows  that $\Sigma$ is the first minimal restriction of $\widehat{\Sigma}.$
\end{proof}

\section{Generalized defect functions} \label{sec7}

If $\cU$ and $\cY$ are Hilbert spaces,   it is well known that $S \in  \mathbf{S}(\cU,\cY)$  is holomorphic in the unit disk and it has
    non-tangential contractive strong limit values almost everywhere (a.e.) on the unit circle $\dT.$ Therefore,  $S$ can be extended to   $L^{\infty}(\cU,\cY)$ function, that is,
    the class of weakly measurable a.e. defined and essentially bounded $\mathcal{L}(\cU,\cY)$-valued functions  on  $\dT.$ Then it follows from  \cite[Theorem V.4.2]{SF} that there exist a Hilbert space $\cK$ and an outer function  $\varphi_S \in \mathbf{S}(\cU,\cK)$ such that \begin{equation}\label{us-def1}\varphi_S^*(\zeta)\varphi_S(\zeta)\leq I-S^*(\zeta)S(\zeta)\end{equation} a.e. on $\dT$, and if a function  $\widehat{\varphi}\in \mathbf{S}(\cU,\widehat{\cK})$, where  $\widehat{\cK}$ is a Hilbert space, has this same property, then   \begin{equation}\label{us-def2}\widehat{\varphi}^*(\zeta)\widehat{\varphi}(\zeta)\leq \varphi_S^*(\zeta)\varphi_S(\zeta)\end{equation} a.e. on $\dT.$ The function $\varphi_S$ is called  the \textbf{right defect function} of $S.$ For the notions of the outer functions, $^*$-outer functions, inner functions and  $^*$-inner functions, see  \cite[Chapter V]{SF}.
From   \cite[Theorem V.4.2]{SF} it is also easy to deduce that there exists a Hilbert space $\cH$ and a $^*$-outer function $\psi_S \in  \mathbf{S}(\cH,\cY)$ such that \begin{equation}\label{us-def3} \psi_S(\zeta) \psi_S^*(\zeta) \leq  I-S(\zeta) S^*(\zeta)\end{equation}  a.e. $\zeta \in\dT$ and if a Schur function $\widehat{\psi} \in \mathbf{S}(\widehat{\cH},\cY)$ has this same property, then \begin{equation}\label{us-def4} \psi_S(\zeta) \psi_S^*(\zeta) \leq  \widehat{\psi}(\zeta) \widehat{\psi}^*(\zeta).\end{equation}
 The function $\psi_S$ is called  the \textbf{left defect function} of $S.$ Both $\varphi_S$ and $\psi_S$ are unique up to a unitary constant.

The theory of the defect functions is  considered, for instance,  in \cite{Boiko97,Boikop1,Boikop2}.  Various connections of defect functions and passive realizations can be found  in \cite{Seppo,ArNu1,ArNu2}. The definition of the defect functions was generalized for  functions $S \in \SK$ in \cite{Lassi} by using the Kre\u{\i}n-Langer factorizations and the fact that all functions  in $\SK$ have also contractive strong limit values a.e. on $\dT.$
If $\cU$ and $\cY$ are Pontryagin spaces such that their negative index  is not zero, the defect functions cannot be defined similarly as in the Hilbert space settings, since $S \in \SK$ may not be extendable to the unit circle. In the Hilbert state space case, Arov and Saprikin showed in \cite{SapAr} that for a function $S=S_rB_r^{-1} \in \SK$, where $S_rB_r^{-1}$ is the right Kre\u{\i}n-Langer factorization of $S$, the existence of the optimal minimal realization of $S$  is connected with the existence of the right defect function of $S_r.$ In general, similar connections exist with  certain functions constructed by embedded systems, and those function are called defect functions; this is the approach taken here.

  Suppose that $\sys$ is a passive realization of $\theta \in \SK.$ Denote the system operator of $\Sigma$ by  $T.$ Theorem \ref{Julia} shows that  there exists a Julia operator of $T.$ By using the same notation as in \eqref{JuliaoperatorT}, one can form the \textbf{Julia embedding} $\widetilde{\Sigma}$ of the system $\Sigma.$ Then the corresponding system operator $\widetilde{T}$ is a Julia operator of $T,$ and it is of the form
   \begin{equation} \label{Juliaemb}T_{\widetilde{\Sigma}} = \begin{pmatrix}
    A & \begin{pmatrix} B & D_{T_{,1}^*} \end{pmatrix} \\
                               \begin{pmatrix}  C \\   D_{T_{,1}}^*  \end{pmatrix} &  \begin{pmatrix}
                                      D &  D_{T_{,2}^*} \\
                                       D_{T_{,2}}^* & -L^*
                                    \end{pmatrix}
                             \end{pmatrix}  : \begin{pmatrix}
                                            \cX \\
                                          \begin{pmatrix}  \cU \\
                                            \sD_{T^*}\end{pmatrix}
                                          \end{pmatrix} \to \begin{pmatrix}
                                            \cX \\\begin{pmatrix}
                                            \cY \\
                                            \sD_{T}\end{pmatrix}
                                          \end{pmatrix},
  \end{equation} where  \begin{equation*} D_{T^*}\! =\! \begin{pmatrix}
                                      D_{T_{,1}^*}\\
                                      D_{T_{,2}^*}
                                    \end{pmatrix}\!,\quad D_{T} \!=\! \begin{pmatrix}
                                      D_{T_{,1}}\\
                                      D_{T_{,2}}
                                    \end{pmatrix}\!,\quad   D_{T^*}D_{T^*}^*\!=\!I_{\cX}-TT^*\!,\quad   D_{T}D_{T}^*\!=\!I_{\cX}-T^*T,  \end{equation*}  such that $D_{T}$ and $D_{T^*}$ have zero kernels. The transfer function of the Julia embedding is
  \begin{equation}\label{embtrans}\begin{split} {\theta_{\widetilde{\Sigma}}}(z)&=
   \begin{pmatrix}
                                                       D + zC(I-zA)^{-1}B & D_{T_{,2}^*}+zC(I-zA)^{-1} D_{T^*_{,1}} \\
                                                          D_{T_{,2}}^* +zD_{T_{,1}}^*(I-zA)^{-1}B &-L^* + zD_{T_{,1}}^*(I-zA)^{-1} D_{T^*_{,1}}
                                                     \end{pmatrix}\\&= \begin{pmatrix}
                                                                                           \theta (z) & \psi(z)\\
                                                                                           \varphi(z) &  \chi(z)
                                                                                         \end{pmatrix}.
  \end{split}\end{equation}  Lemma \ref{ker-es} shows that the identities
\begin{align}
 I- \theta(z) \theta^*(w)&=(1 -z\bar{w}) G(z)G^*(w)  +\psi(z) \psi^*(w),\label{lefdef}\\    I-\theta^* (w) \theta(z)&=(1 -z\bar{w}) F^*(w)F(z)  +\varphi^*(w) \varphi(z)\label{rightdef}
\end{align}
  where $$G(z)=C(I_{\cX}-zA)^{-1}, \quad F(z)=(I_{\cX}-zA)^{-1}B,$$ holds for every $z$ and $w$ in a sufficiently small symmetric neighbourhood of the origin.    If the negative index of $\cU$ and $\cY$ is denoted as $\kappa_1,$ a similar argument as in  the proof of Proposition \ref{maxdim} shows that the kernels $K_{\varphi}(w,z)$  and $K_{\psi}(w,z)$ defined as in \eqref{kernel1} have at most $\kappa + \kappa_1$ negative squares.

  \begin{definition}\label{defidef} Let $\cU$ and $\cY$ be Pontryagin spaces with the same negative index.
    Let $\sys$ be an optimal minimal passive realization of $\theta \in \SK,$ and let  $\widetilde{\Sigma}$ be the Julia embedding of it, represented as in  \eqref{Juliaemb}. 
    Then the function $\varphi$ in  \eqref{embtrans} is defined to be the right defect function $\varphi_\theta$ of $\theta.$ 

    Moreover, let  $\sys$ be a $^*$-optimal minimal  passive realization of $\theta \in \SK,$ and let  $\widetilde{\Sigma}$ be the Julia embedding of it, represented as in \eqref{Juliaemb}. 
     Then the function $\psi$ in \eqref{embtrans} is defined to be the left defect function $\psi_{\theta}$ of $\theta.$ 
  \end{definition}
 \begin{remark}\label{remarkdef}
 Since optimal ($^*$-optimal) minimal  realizations are unitarily similar by Theorem \ref{optimals}, and Julia operators for contractive operator are essentially unique by Theorem \ref{Julia}, it can be deduced that the defect functions are essentially uniquely defined by $\theta \in \SK.$  The definition above is also slightly different from the one given in \cite{Lassi} for functions in the class $\SK,$ where  $\cU$ and $\cY$ are Hilbert spaces.
 \end{remark}

The right defect function of $\theta \in \SK$ and the left defect function of $\theta^\#$ are closely related to each other.
\begin{lemma}\label{related}
For  $\theta \in \SK,$  it holds $\varphi_{\theta}^\#=\psi_{\theta^\#}$ and $\psi_\theta^\#=\varphi_{\theta^\#}$
\end{lemma}
\begin{proof}
  Let $\sys$ be an optimal ($^*$-optimal) minimal realization of $\theta.$ Denote the system operator of $\Sigma$ as $T,$ and the Julia operator $T_{\widetilde{\Sigma}}$ of $T$ as in \eqref{Juliaemb}. By Theorem \ref{optimals}, the system $\Sigma^*$ is $^*$-optimal  (optimal) minimal, and a calculation shows that $T_{\widetilde{\Sigma}}^*$ is the Julia operator of $T^*.$ Now the results follow means of \eqref{embtrans}.
\end{proof}

In the Hilbert space settings, $S \in \mathbf{S}(\cU,\cY) $ has factorizations of the form
$$S=S_i S_o= S_{*o}  S_{*i},   $$ where $S_i \in \mathbf{S}(\cY,\cY) $ is inner,  $S_o \in \mathbf{S}(\cU,\cY) $ is outer,  $S_{*o} \in \mathbf{S}(\cU,\cY) $ is $^*$-outer and
$S_{*i} \in \mathbf{S}(\cY,\cY) $ is $^*$-inner  \cite[p. 204]{SF}. The next proposition  shows that for an ordinary Schur function $ \theta \in \mathbf{S}(\cU,\cY), $ the outer factor of $\varphi_\theta$ and the $^*$-outer factor of $\psi_\theta$ defined above coincide essentially with the usual definition of defect functions.
\begin{proposition}\label{larmino}
Let $\theta \in\SK,$ where $\cU$ and $\cY$ are Hilbert spaces. 
Then  
 \begin{equation}\label{contrp}\varphi_\theta^*(\zeta)\varphi_\theta(\zeta)\leq I-\theta^*(\zeta)\theta(\zeta)\end{equation} a.e. on $\dT$, and if a generalized Schur function  $\widehat{\varphi}\in \mathbf{S}_{\kappa'}(\cU,\widehat{\cK})$, where  $\widehat{\cK}$ is a Hilbert space and $\kappa'$ does not depend on $\kappa$, has this same property, then  \begin{equation}\label{maximalp}\widehat{\varphi}^*(\zeta)\widehat{\varphi}(\zeta)\leq \varphi_\theta^*(\zeta)\varphi_\theta(\zeta),\end{equation} a.e. on $\dT.$ If $\kappa=0,$ denote the inner and outer factors of $\varphi_\theta$ as  $\varphi_{\theta_i}$ and $\varphi_{\theta_o},$ respectively. Then, $\varphi_{\theta_i}$ is an isometric constant,  and if  $\varphi'$ is  an outer function with properties \eqref{us-def1} and  \eqref{us-def2}, then it holds $U\varphi_{\theta_o}=\varphi',$  where $U$ is a unitary operator.

 Moreover, 
  $$ \psi_\theta(\zeta) \psi_\theta^*(\zeta) \leq  I-\theta(\zeta) \theta^*(\zeta)$$  a.e. $\zeta \in\dT$ and if a generalized Schur function $\widehat{\psi} \in \mathbf{S}_{\kappa'}(\widehat{\cH},\cY)$,  where  $\widehat{\cK}$ is a Hilbert space and $\kappa'$ does not depend on $\kappa$, has this same property, then $$ \psi_\theta(\zeta) \psi_\theta^*(\zeta) \leq  \widehat{\psi}(\zeta) \widehat{\psi}^*(\zeta)$$ a.e. $\zeta \in\dT. $If $\kappa=0,$ denote the $^*$-inner and $^*$-outer factors of $\psi_\theta$ as  $\psi_{\theta_{^*i}}$ and $\psi_{\theta_{^*o}},$ respectively. Then, $\psi_{\theta_{^*i}}$ is a co-isometric constant,  and if  $\psi'$ is  a $^*$-outer function with properties \eqref{us-def3} and  \eqref{us-def4}, then it holds $\psi_{\theta_{^*o}}U'=\psi',$  where $U'$ is a unitary operator.
\end{proposition}
\begin{proof}
Let \begin{equation}\label{sig}\sys\end{equation} be an optimal minimal realization of $\theta.$ Denote the system operator of $\Sigma$ as $T,$ the Julia operator $T_{\widetilde{\Sigma}}$ of $T$  as in \eqref{Juliaemb} and the function $\varphi=\varphi_\theta$ as in \eqref{embtrans}. Since $T_{\widetilde{\Sigma}}$ is unitary, the operator \begin{equation}\label{embeiso}
                                                                                                  T_{{\Sigma'}} =  \begin{pmatrix}
                                                                                                    A &  B  \\
                                                                                                       \begin{pmatrix}
                                                                                                                       C \\
                                                                                                                     D_{T_{,1}}^*
                                                                                                                    \end{pmatrix} & \begin{pmatrix}
                                                                                                                     D \\
                                                                                                                      D_{T_{,2}}^*
                                                                                                                    \end{pmatrix}
                                                                                                     \end{pmatrix}: \begin{pmatrix}
                                                                                                                      \cX \\
                                                                                                                      \cU
                                                                                                                    \end{pmatrix}  \to \begin{pmatrix}
                                                                                                                      \cX \\
                                                                                                                      \begin{pmatrix}
                                                                                                   \cY \\
                                                                                                     \sD_{T}
                                                                                                   \end{pmatrix}
                                                                                                                    \end{pmatrix}.
                                                                                                   \end{equation} 
  must be isometric, and therefore the system $$\Sigma'=\left(A,B,\begin{pmatrix}
                                                                                                                       C \\
                                                                                                                     D_{T_{,1}}^*
                                                                                                                    \end{pmatrix}, \begin{pmatrix}
                                                                                                                     D \\
                                                                                                                      D_{T_{,2}}^*
                                                                                                                    \end{pmatrix};\cX,\cU, \begin{pmatrix}
                                                                                                   \cY \\
                                                                                                     \sD_{T}
                                                                                                   \end{pmatrix};\kappa\right)$$ is an isometric realization of  the function  { \tiny $\begin{pmatrix}\theta\\\varphi_\theta\end{pmatrix} $ }. Since $\Sigma'$ is an embedding of the minimal system $\Sigma,$ the system  $\Sigma'$ is also minimal. It follows from Theorem \ref{realz} that
                                                                                                  { \tiny $\begin{pmatrix}\theta\\\varphi_\theta\end{pmatrix}$} $\in \mathbf{S}_\kappa\left(\cU,
                                                                                                   \cY \oplus
                                                                                                     \sD_{T}
                                                                                                    \right).$ Since contractive boundary values of  generalized Schur functions exist for a.e. $\zeta \in\dT,$ it holds
                                                                                                   \begin{equation*}
                                                               \begin{pmatrix}\theta^*(\zeta) & \varphi_\theta^*(\zeta)\end{pmatrix}                    \begin{pmatrix}\theta(\zeta)\\\varphi_\theta(\zeta)\end{pmatrix}                  \leq I \qekv \varphi_\theta^*(\zeta)\varphi_\theta(\zeta)\leq I-\theta^*(\zeta)\theta(\zeta)
                                                                                                   \end{equation*} for a.e. $\zeta \in \dT.$

Suppose that a function  $\widehat{\varphi}\in \mathbf{S}_{\kappa'}(\cU,\widehat{\cK})$, where  $\widehat{\cK}$ is a Hilbert space, has the property
 \begin{equation*}
 \widehat{\varphi}^* (\zeta) \widehat{\varphi}(\zeta) \leq I-\theta^* (\zeta) \theta(\zeta)
\end{equation*} for a.e. $\zeta\in \dT.$ Since the function $\widehat{\varphi}$ has the left Kre\u{\i}n--Langer factorization of the form $\widehat{\varphi}= B^{-1}_{\widehat{\varphi}} \widehat{\varphi}_{l} 
,$ where $\widehat{\varphi}_{l} $ is an ordinary Schur function,    
 it holds
$$   \widehat{\varphi}^* (\zeta) \widehat{\varphi}(\zeta) = \widehat{\varphi}_{l}^* (\zeta) \widehat{\varphi}_{l}(\zeta)    $$ for a.e. $\zeta\in \dT.$
Then the function \begin{equation}\label{u-hat} \breve{\theta} = \begin{pmatrix}
                                                            \theta \\
                                                             \widehat{\varphi}_{l}
                                                          \end{pmatrix},\end{equation} belongs to the Schur class $\mathbf{S}_\kappa\left(
                                                                                                     \cU
                                                                                                   ,
                                                                                                   \cY \oplus
                                                                                                     \widehat{\cK}
                                                                                                   \right),$ and it has a controllable isometric  realization $\breve{{\Sigma}}$ with the system operator
                                                                                                   \begin{equation*}
                                                                                                  T_{\breve{\Sigma}} =  \begin{pmatrix}
                                                                                                       A_1 & B_1 \\
                                                                                                       \begin{pmatrix}
                                                                                                                      C_1 \\
                                                                                                                      C_2
                                                                                                                    \end{pmatrix} & \begin{pmatrix}
                                                                                                                      D_1 \\
                                                                                                                      D_2
                                                                                                                    \end{pmatrix}
                                                                                                     \end{pmatrix}: \begin{pmatrix}
                                                                                                                      \cX_1 \\
                                                                                                                      \cU
                                                                                                                    \end{pmatrix}  \to \begin{pmatrix}
                                                                                                                      \cX_1 \\
                                                                                                                      \begin{pmatrix}
                                                                                                   \cY \\
                                                                                                     \widehat{\cK}
                                                                                                   \end{pmatrix}
                                                                                                                    \end{pmatrix}.
                                                                                                   \end{equation*} That is,
                                                                                                   \begin{align*}
                                                          \breve{\theta}(z)&=\begin{pmatrix}
                                                            \theta(z) \\
                                                             \widehat{\varphi}_{l}(z)                                                          \end{pmatrix}=\begin{pmatrix}
                                                                                                                      D_1 \\
                                                                                                                      D_2
                                                                                                                    \end{pmatrix} + z \begin{pmatrix}
                                                                                                                      C_1 \\
                                                                                                                      C_2
                                                                                                                    \end{pmatrix}(I-zA_1)^{-1}B_1\\&=\begin{pmatrix}
                                                                                                                                                  D_1 + z C_1(I-zA_1)^{-1}B_1 \\
                                                                                                                                                    D_2 + z C_2(I-zA_1)^{-1}B_1
                                                                                                                                                  \end{pmatrix}.
                                                                                                   \end{align*} It follows that \begin{equation}\label{sigma}\Sigma_1=(A_1,B_1,C_1,D_1;\cX_1,\cU,\cY;\kappa)\end{equation} is a realization of $\theta,$ and since ${\breve{\Sigma}}$ is isometric and $\widehat{\cK}$ is a Hilbert space, the system $\Sigma_1$ is  passive. Since $T_{\breve{\Sigma}}$ is isometric, the defect operator $D_{T_{\breve{\Sigma}}}$ of $T_{\breve{\Sigma}}$ is zero, and it follows from Lemma \ref{ker-es} that
                                                                                                     \begin{equation}\label{isom} \begin{split}I-\breve{\theta}^* (z) \breve{\theta}(z)&= I-{\theta}^* (z) {\theta}(z)- \widehat{\varphi}_{l}^* (z)\widehat{\varphi}_{l}(z)\\&=\left(1 -|z|^2\right) B_1^*(I-\overline{z}A_1^*)^{-1})(I-zA_1)^{-1}B_1 \end{split}\end{equation} whenever the expressions are meaningful.    By combining the identities \eqref{rightdef}  and \eqref{isom} for optimal minimal realization $\Sigma,$ one gets
\begin{equation}\begin{split}\label{isom2}
\left(1 -|z|^2\right) &B_1^*(I-\overline{z}A_1^*)^{-1}(I-zA_1)^{-1}B_1+\widehat{\varphi}_{l}^* (z) \widehat{\varphi}_{l}(z)\\
&=\left(1 -|z|^2\right) B^*(I-\overline{z}A^*)^{-1}(I-zA)^{-1}B+{\varphi_\theta}^* (z){\varphi_\theta}(z)
 \end{split}\end{equation}  for every $z$ in a sufficiently small symmetric neighbourhood $\Omega$ of the origin. Since the system $\Sigma$ is optimal, if follows by using Neumann series that
 \begin{align*}
  \left\langle \!B^*(I-\overline{z}A^*)^{-1}\!(I-zA)^{-1}\!Bu,u \!\right\rangle&\!= E_{\cX}\!\!\left( (I-zA)^{-1}Bu  \right) = E_{\cX}\!\!\left(\sum_{n=0}^{\infty} A^{n}Buz^n  \!\!\right)\\
  &\leq  \!E_{\cX_1}\!\!\left(\sum_{n=0}^{\infty} A_1^{n}B_1uz^n\!  \right) \\&= \left\langle B_1^*(I-\overline{z}A_1^*)^{-1}\!(I-zA_1)^{-1}\!B_1u,u \right\rangle
 \end{align*} for every $z \in \Omega$ and for every $u \in \cU.$ Then it follows from \eqref{isom2} that
 $$     \widehat{\varphi}_{l}^* (z) \widehat{\varphi}_{l}(z) \leq {\varphi_\theta}^* (z){\varphi_\theta}(z), \qquad z \in \Omega.  $$ By continuity,
  \begin{equation}\label{mino-om}   \widehat{\varphi}_{l} ^* (\zeta) \widehat{\varphi}_{l}(\zeta)= \widehat{\varphi}^* (\zeta) \widehat{\varphi}(\zeta) \leq  {\varphi_\theta}^* (\zeta){\varphi_\theta}(\zeta)    \end{equation} for a.e. $\zeta \in \dT.$

Next suppose  that $\kappa=0.$ By combining \eqref{us-def2} and \eqref{mino-om}, it can be deduced that
$$   {\varphi^{'*}}(\zeta){\varphi'}(\zeta)= {\varphi_\theta}^* (\zeta){\varphi_\theta}(\zeta)= {\varphi_{\theta_o}}^*(\zeta) {\varphi_{\theta_i}}^*(\zeta) {\varphi_{\theta_i}} (\zeta){\varphi_{\theta_o}}(\zeta) ={\varphi_{\theta_o}}^* (\zeta){\varphi_{\theta_o}}(\zeta)$$ for a.e. $\zeta \in \dT.$ Then it follows from \cite[Proposition V.4.1]{SF} that $\varphi'=U\varphi_{\theta_o},$ where $U$ is  a unitary operator. If one puts an outer function $ \widehat{\varphi}_{l}=\varphi_{\theta_{o}}=U^{-1}\varphi'$ in \eqref{u-hat}, the construction of an optimal minimal system used in the proof of \cite[Theorem 7]{Arov}  shows that the associated system $\Sigma_1$ in \eqref{sigma} is optimal. Since $\Sigma$ in \eqref{sig} is also optimal,  for every $z \in \dD,$ it holds
$$   B^*(I-\overline{z}A^*)^{-1}(I-zA)^{-1}B = B_1^*(I-\overline{z}A_1^*)^{-1}(I-zA_1)^{-1}B_1,   $$ and then by \eqref{isom2}
$$\|  {\varphi_{\theta_i}} (z){\varphi_{\theta_o}}(z)u \| = \|  {\varphi_{\theta_o}}(z)u \|
   $$
    for every $z \in \dD$ and every $u \in \cU.$ The outer function ${\varphi_{\theta_o}}(z)$ has a dense range for every $z \in \dD$  \cite[Proposition V.2.4]{SF}. This implies that $ {\varphi_{\theta_i}} (z)$ is an isometry for every  $z \in \dD,$ and arguing as in the proof of \cite[Proposition V.2.1]{SF} one deduces that $ {\varphi_{\theta_i}} $ is an isometric constant.
The  claims involving $ {\varphi_{\theta}} $ are proved.

  The  claims involving $ {\psi_{\theta}} $ follow now directly by applying Lemma \ref{related}.
\end{proof}
\begin{lemma} \label{suf-opti} Let $$\Sigma_0=(A_0,B_0,C_0,D;\cX_0,\cU,\cY;\kappa), \quad\sys$$  be passive   realizations of $\theta \in \SK$  such that $\Sigma_0$ is  optimal.  If
  for every $z$ and $w$ in a sufficiently small symmetric neighbourhood $\Omega$ of the origin the equality
  \begin{equation} \label{sufnes-opti}
    B^*(I-\overline{w}A^*)^{-1}(I-zA)^{-1}B = B_0^*(I-\overline{w}A_0^*)^{-1}(I-zA_0)^{-1}B_0
  \end{equation}  holds, then $\Sigma$ is optimal. 
   \end{lemma}
\begin{proof} It follows from  Lemma \ref{repres} that 
 the system operator $T_\Sigma$ of $\Sigma$ can be represented as in \eqref{rep-contro}, the restriction $\Sigma_c=(A_c,B_c,C_c,D;\cX^c,\cU,\cY;\kappa)$ of $\Sigma$ to the controllable subspace $\cX^c$ is controllable passive, and \eqref{contser} and \eqref{contser1} hold.

Let \begin{equation}\label{ex} x=\sum_{j=1}^M A_c^j B_cu_j, \quad M \in \dN, \quad \{u_j\}_{j=1}^M \subset \cU.\end{equation} 
  For the vectors of the form \eqref{ex}, define
  \begin{equation}\label{Rex} Rx=\sum_{j=1}^M A_0^j B_0u_j.\end{equation}
 Since  $\Sigma_c$ is controllable,  the domain of $R$ is dense. 
  Moreover, $\Sigma_0$ is optimal, and therefore
 $ E_{\cX_0} \left( Rx \right) \leq   E_{\cX^c} \left( x \right)   .  $ That is, $R$ is contractive, and it follows from \cite[Theorem 1.4.2]{ADRS} that the closure of $R$ is everywhere defined contractive linear operator. It is still denoted by $R.$ Since
 $$ (I-zA_c)^{-1}B_c=\sum_{n=0}^\infty z^nA_c^nB_c,   \qquad      (I-zA_0)^{-1}B_0=\sum_{n=0}^\infty z^nA_0^nB_0,   $$ holds  for every $z$ in a sufficiently small symmetric neighbourhood $\Omega$ of the origin,  it follows by continuity that
 $$  R\left((I-zA_c)^{-1}B_cu\right)  =(I-zA_0)^{-1}B_0u $$ for every $z\in \Omega$ and $u\in \cU.$ Then
 $$  R\left(  \sum_{j=1}^M (I-z_jA_c)^{-1}B_cu_j    \right)          =         \sum_{j=1}^M (I_{\cX_0}-z_jA_0)^{-1}B_0u_j ,  $$ for all $ M \in \dN,  \{z_j \}_{j=1}^M \subset \Omega,  $ and $   \{u_j \}_{j=1}^M \subset \cU.$
 Equalities \eqref{contser1} and  \eqref{sufnes-opti} imply now
 \begin{equation*}\begin{split}
 &E_{\cX^c} \left(    \sum_{j=1}^M (I-z_jA_c)^{-1}B_cu_j       \right)  \\&\quad=\sum_{j=1}^M \sum_{k=1}^M \left\langle   B_c^* (I-\overline{z_k}A_c^*)^{-1}(I-z_jA_c)^{-1}B_cu_j   ,u_k \right\rangle_{\cU} \\
 &\quad=\sum_{j=1}^M \sum_{k=1}^M \left\langle   B_0^* (I-\overline{z_k}A_0^*)^{-1}(I-z_jA_0)^{-1}B_0u_j   ,u_k \right\rangle_{\cU} \\
 &\quad=E_{\cX_0} \left(    \sum_{j=1}^M (I-z_jA_0)^{-1}B_0u_j       \right)\\&\quad= E_{\cX_0} \left( R\left(   \sum_{j=1}^M (I-z_jA_c)^{-1}B_cu_j    \right)     \right).
\end{split}.\end{equation*}  This implies that $R$ is isometric in $\mathrm{span}\{\ran (I-zA_1)^{-1}B_1, z \in \Omega\},$ which is a dense set, since $\Sigma_1$ is controllable. Since $R$ is bounded, it is now isometric everywhere, and identity \eqref{Rex} implies that $\Sigma_c$ is optimal. Then it follows from \eqref{contser} that $\Sigma$ is optimal, and the proof is complete.
 \end{proof}

The main results of \cite[Theorem 1.1]{Seppo} were generalized to the Pontryagin state space setting in \cite[Theorem 4.4]{Lassi}. By using Definition   \ref{defidef}, it can be shown that  parts of this result, as well as \cite[Theorem 1]{ArNu2}, hold also in the case when all the spaces are indefinite. Moreover, certain parts of \cite[Theorem 1.1]{Seppo}, \cite[Theorem 1]{ArNu2} and \cite[Theorem 4.4]{Lassi}   can be improved. Before stating these results,  some lemmas are needed.
\begin{lemma}\label{equivopti} Let $\theta \in \SK$. Then the  following statements are equivalent:
\begin{itemize}
  \item[\textrm{(i)}] all $\kappa$-admissible minimal passive realizations of $\theta$  are unitarily similar;
  \item[\textrm{(ii)}] there exists a minimal passive realization of $\theta$ such that it is both optimal and $^*$-optimal;
  \item[\textrm{(iii)}] all $\kappa$-admissible minimal passive realizations of $\theta$  are both optimal and $^*$-optimal.
\end{itemize}
\end{lemma}
\begin{proof}
  \textrm{(i)} $\Rightarrow$  \textrm{(iii)}.  Suppose  \textrm{(i)}.  Let $$\Sigma_1=(A_1,B_1,C_1,D;\cX_1,\cU,\cY;\kappa), \quad \Sigma_2=(A_2,B_2,C_2,D;\cX_2,\cU,\cY;\kappa)$$  be, respectively,  minimal passive and optimal  ($^*$-optimal)  minimal passive realizations of $\theta.$ Let $U$ be the unitary operator from $\cX_1$ to $\cX_2$ with the properties described in \eqref{unisim}. An easy calculation shows that
  $$    E_{\cX_2}\left(    \sum_{k=0}^{n}A_2^kB_2u_k    \right) =   E_{\cX_1}\left( U   \sum_{k=0}^{n}A_1^kB_1u_k    \right) =  E_{\cX_1}\left(    \sum_{k=0}^{n}A_1^kB_1u_k    \right)    $$ for every $u \in \cU$ and for every $n=0,1,2,\ldots$ which implies that $\Sigma_1$ is actually optimal ($^*$-optimal), 
  and  therefore \textrm{(iii)} holds.

 \textrm{(iii)}  $\Rightarrow$   \textrm{(ii)}.  The claim  \textrm{(iii)} trivially implies \textrm{(ii)}.

 \textrm{(ii)}  $\Rightarrow$   \textrm{(i)}.  Suppose \textrm{(ii)}. Let the systems $$\Sigma_1=(A_1,B_1,C_1,D;\cX_1,\cU,\cY;\kappa), \quad \Sigma_2=(A_2,B_2,C_2,D;\cX_2,\cU,\cY;\kappa)$$  be   minimal passive realizations  of $\theta$ such that $\Sigma_1$ is optimal and $^*$-optimal.  Let $Z$ be the weak similarity mapping  from $\cX_1$ to $\cX_2$  with the  properties described in \eqref{weaksim}. It follows from \eqref{weaksim} that all elements of the form $\sum_{k=0}^nA_1^kB_1u_k$ belongs to the domain of $Z,$ and $Z\left(\sum_{k=0}^nA_1^kB_1u_k\right)=\sum_{k=0}^nA_2^kB_2u_k$. Recall also here the construction of $Z$ in the proof of \cite[Theorem 2.5]{Lassi}.  Since $\Sigma_1$ is both optimal and $^*$-optimal,
  $$    E_{\cX_2}\left(    \sum_{k=0}^{n}A_2^kB_2u_k    \right) =   E_{\cX_2}\left( Z  \sum_{k=0}^{n}A_1^kB_1u_k    \right) =  E_{\cX_1}\left(    \sum_{k=0}^{n}A_1^kB_1u_k    \right).$$ Then it follows from  \cite[Theorem 1.4.2]{ADRS} that the operator $Z$ has a unitary extension, and the properties in \eqref{unisim} follow by continuity. Therefore $\Sigma_1$ and $\Sigma_2$ are unitarily similar. Since unitary similarity clearly is a transitive property,  \textrm{(i)} holds, and the proof is complete.
\end{proof}

\begin{lemma}\label{opti-cont}
  If the system $\sys$ is an optimal passive realization of $\theta \in \SK,$ then $\cX^c \subset \cX^o.$
\end{lemma}
\begin{proof}
  According to Proposition \ref{OnHilb}, the spaces $\cX^o$ and $(\cX^o)^\perp$ are regular subspaces and $(\cX^o)^\perp$ is a Hilbert space. It follows from Lemma \ref{repres} that the system operator $T$ of $\Sigma$ can be represented as in \eqref{rep-obse},
  %
  %
   and the restriction      $\Sigma_o=(A_o,B_o,C_o,D;\cX^o,\cU,\cY;\kappa)$  of $\Sigma$ to the observable subspace $\cX^o$ is observable  passive realization of $\theta.$ For $n=0,1,2,\ldots$, it holds $$A^n=  \begin{pmatrix}
          A_{1}^n & f(n) \\
          0 & A_0^n
        \end{pmatrix},  $$ where $f(n)$ is an operator depending on $n$. 
        Then for any $N\in \dN_0$ and any $\{u_n \}_{n=0}^N \subset \cU,$ it holds
        \begin{align*}
          \sum_{n=0}^{N} A^nBu_n
                         \!=\!\begin{pmatrix}
                       \sum_{n=0}^{N}\left(   A_{1}^n B_1u_n +f(n) B_ou_n\right)\\
                         \sum_{n=0}^{N}  A_o^nB_ou_n
                        \end{pmatrix}
                       \!\! =\!\!   \begin{pmatrix}
                       P_{(\cX^o)^\perp}\left(  \sum_{n=0}^{N} A^nBu_n\right)\\
                        P_{\cX^o} \left(\sum_{n=0}^{N} A^nBu_n\right)
                        \end{pmatrix}.
        \end{align*} This implies
         \begin{align*}
         E_{\cX}\!\!\left( \sum_{n=0}^{N} A^nBu_n \! \right) 
                        \!=&   E_{(\cX^o)^\perp}\left(     P_{(\cX^o)^\perp}\left(\sum_{n=0}^{N} A^nBu_n\right)   \right)  + E_{\cX^o}\left(\sum_{n=0}^{N} A_o^nB_ou_n \!\right).
        \end{align*} But since $\Sigma$ is optimal  and $(\cX^o)^\perp$ is a Hilbert space, one deduces that $$  P_{(\cX^o)^\perp}\left(\sum_{n=0}^{N} A^nBu_n\right) =0.$$ That is, $\mathrm{span}\{ A^nB : n=0,1,\ldots \} \subset \cX^o$ and since $\cX^o$ is closed, also $$\overline{\mathrm{span}}\{ A^nB : n=0,1,\ldots \} = \cX^c \subset \cX^o.$$ 
  \end{proof}


The next Theorem contains promised extensions for some results of \cite{Seppo}. In particular, the fact that statements  \textrm{(I)(b), (II)(b)} and
\textrm{(III)(b)} implies the other statements, respectively, in parts \textrm{(I), (II)} and
\textrm{(III)}, is new also in the Hilbert space setting.
\begin{theorem} \label{sepontulos} Let $\theta \in \SK,$ where $\cU$ and $\cY$ are Pontryagin spaces with the same negative index.
  \begin{enumerate}
    \item[\textrm{(I)}] The following statements are equivalent:
              \begin{enumerate}
              \item[\textrm{(a)}] $\varphi_\theta\equiv0$;
              \item[\textrm{(b)}] all  $\kappa$-admissible controllable passive realizations of $\theta$ are minimal isometric;
              \item[\textrm{(c)}] there exists an  observable conservative  realization of $\theta;$
               \item[\textrm{(d)}]  all simple conservative realization of $\theta$ are observable;
                \item[\textrm{(e)}] all  observable co-isometric realizations of $\theta$ are conservative.
               \end{enumerate}
    \item[\textrm{(II)}]The following statements are equivalent:
               \begin{enumerate}
              \item[\textrm{(a)}] $\psi_\theta\equiv0$;
              \item[\textrm{(b)}] all  $\kappa$-admissible observable  passive realization of $\theta$ are minimal co-isometric;
              \item[\textrm{(c)}] there exists a controllable conservative   realization of $\theta;$
               \item[\textrm{(d)}]  all simple conservative realization of $\theta$ are controllable;
                \item[\textrm{(e)}] all  controllable  isometric  realizations of $\theta$ are conservative.
               \end{enumerate}
    \item[\textrm{(III)}] The following statements are equivalent:
              \begin{enumerate}
              \item[\textrm{(a)}]  $\varphi_\theta\equiv0$ and $\psi_\theta\equiv0$;
               \item[\textrm{(b)}]  all  $\kappa$-admissible simple passive realization of $\theta$ are minimal conservative;
                \item[\textrm{(d)}] there exists a  minimal conservative realization of $\theta.$
               \end{enumerate}
  \end{enumerate}
\end{theorem}
\begin{proof}
  \textrm{(I)}    \textrm{(a)}  $\Rightarrow$  \textrm{(b)}. Suppose \textrm{(a)}. Let the systems $$\sys, \quad\Sigma_0=(A_0,B_0,C_0,D;\cX_0,\cU,\cY;\kappa)$$ be, respectively,     a controllable  passive and an optimal minimal passive realizations  of $\theta.$ Represent the Julia embeddings of $\Sigma$ and $\Sigma_0$ as in \eqref{Juliaemb}.
  Then, \eqref{rightdef} holds for $\Sigma$. Since $\varphi_\theta\equiv0,$ if follows from the definition of $\varphi_\theta$ that
  $$  I-\theta^*(w) \theta(z)=  (1-z\bar{w})B_0^*(I-\bar{w}A^*_0)^{-1}(I-zA_0)^{-1}B_0   $$ holds for every $z$ and $w$ in a sufficiently small symmetric neighbourhood $\Omega$ of the origin. Since $\Sigma_0$ is optimal, by considering the Neuman series of $(I-zA_0)^{-1}B_0$ and $(I-zA_0)^{-1}B_0,$ one deduces that
  $$    B_0^*(I-\bar{z}A^*_0)^{-1}(I-zA_0)^{-1}B_0 \leq B^*(I-\bar{z}A^*)^{-1}(I-zA)^{-1}B, \qquad z \in \Omega.         $$
  Then it holds $\varphi^*(z) \varphi(z) \leq  =0   $  for every $z \in \Omega$. But since $\varphi(z)$ is an operator whose range belongs to the Hilbert space $\mathfrak{D}_T,$ this implies $$\varphi(z)= D_{T_{,2}}^* +zD_{T_{,1}}^*(I-zA)^{-1}B = 0, \qquad z \in \Omega.$$ It follows that $ D_{T_{,2}}^*=0.$ Since $\Sigma$ is controllable, $\mathrm{span}\{(I-zA)^{-1}B; z \in \Omega\}$ is dense in $\cX,$ and therefore also $D_{T_{,1}}^*=0.$ Then   $D_T=0,$ so $T$ is  isometric, and $\Sigma$ is a controllable isometric system.
In particular, if $\Sigma$ is chosen to be minimal passive; for the existence, see Lemma \ref{repres}, the previous argument shows that $\Sigma$ is a  minimal isometric realization of $\theta.$
Since all  controllable isometric realizations of $\theta$ are unitarily similar, they  are now also minimal, and \textrm{(b)} holds.

 \textrm{(b)}  $\Rightarrow$  \textrm{(c)}. Suppose \textrm{(b)}.
Let $\Sigma'=(A',B',C',D;\cX',\cU,\cY;\kappa)
 $  be an
  optimal minimal passive  realization of $\theta.$ The existence of $\Sigma'$ follows from Theorem \ref{optimals} (i). By assumption, $\Sigma'$ is isometric.   It follows from Theorem \ref{optimals} (iv) that $\Sigma'$ is the first minimal restriction of the simple conservative system $\sys$.  By Lemma \ref{repres}, the system operator $T_{\Sigma}$ of $\Sigma$ can be represented as in \eqref{rep-fmini}, where now $\cX'=\overline{P_{\cX^o}\cX^c}$.
   Since the system operator $$   T_{\Sigma'}=  \begin{pmatrix}
          A' &
                         B' \\
          C'
    & D
     \end{pmatrix}: \begin{pmatrix}
                        \cX'
    \\ \cU
                    \end{pmatrix} \to \begin{pmatrix}
                        \cX'
                     \\ \cY
                    \end{pmatrix},$$ of $\Sigma'$  is isometric and $T_{\Sigma}$ is unitary, an easy calculation using  the fact that the range space $(\cX^o)^\perp$ is a Hilbert space shows that $B_1'=0$ and $A_{12}'=0$ in \eqref{rep-fmini}. But then for every $x \in (\cX^o)^\perp$ and every $n=0,1,2,\ldots,$
                    \begin{equation*}
                  B^*A^{*n}x= \begin{pmatrix}
                         0 &
                         {B'}^* &
                         0
                       \end{pmatrix} \begin{pmatrix}
                        {A_{11}'}^* &0&0\\
                        0&{A_{0}'}^* & 0\\
                       {A_{13}'}^* &  {A_{23}'}^* &{A_{33}'}^*
       \end{pmatrix}^n \begin{pmatrix}
                         x \\ 0 \\ 0
                       \end{pmatrix}=0.
                    \end{equation*}
   That is,
   $(\cX^o)^\perp \subset (\cX^c)^\perp$ and therefore   $\cX^c \subset \cX^o$.  Since $\Sigma$ is simple, this implies now $\cX^o=\cX.$ Then $\Sigma$ is observable, and \textrm{(c)} holds. 

 \textrm{(c)}  $\Rightarrow$  \textrm{(a)}. Suppose \textrm{(c)}.
 Let $\sys$  be an observable conservative realization  of $\theta.$  By Lemma \ref{repres}, $\Sigma$ can be represented as in \eqref{rep-fmini}. The first minimal restriction \eqref{res-fmini} of $\Sigma$  is an optimal minimal realization of $\theta$ by Theorem \ref{optimals} (i). But since $\Sigma$ is observable, $\cX^o=\cX$ and $(\cX^o)^\perp=\{0\}.$  It follows  that the reprentations \eqref{rep-contro} and \eqref{rep-fmini} coinsides. That is, the first minimal restriction  $\Sigma'$ is  just a restriction to the controllable subspace of $\Sigma.$  By Lemma \ref{repres}, $\Sigma'$ is now isometric.
                  Thus if one constructs a Julia operator of $T_{\Sigma'}$ as in \eqref{JuliaoperatorT}, $D_{T_{\Sigma'}}=0,$ and then it follows from the definition of $\varphi_{\theta}$ and \eqref{embtrans} that  $\varphi_{\theta}\equiv0,$ and \textrm{(a)} holds.


The equivalences of the statements \textrm{(c)}, \textrm{(d)} and \textrm{(e)} follow easily from the facts that all observable co-isometric realizations of $\theta$  are unitarily similar, all simple conservative realization of $\theta$  are unitarily similar and unitary similarity preserves the structural properties of the system and system operator.  The part  \textrm{(I)} is proven.

   \textrm{(II)} The proof is analogous to the proof of the part \textrm{(I)}, and the details are omitted.

    \textrm{(III)} \textrm{(a)}  $\Rightarrow$  \textrm{(b)}.
     Suppose \textrm{(a)}.  By combining the parts   \textrm{(I)} and   \textrm{(II)}, it follows that all controllable or observable passive realizations of $\theta$ are minimal conservative. Consider a simple  passive realization  $\Sigma\!=(A.B,C,D;\cX,\cU,\cY;\kappa)$ of $\theta.$  It follows from Lemma \ref{repres} that
  the contractive system operator $T$ of $\Sigma$ can be represented as in \eqref{rep-obse},
                   where the restriction $\Sigma_o$ in \eqref{res-obse}  
                   is  observable passive, and therefore now minimal conservative. Then the  system operator
                    $$   T_{\Sigma_o}=  \begin{pmatrix}
          A_o &
                         B_o \\
          C_o
    & D
     \end{pmatrix}: \begin{pmatrix}
                        \cX^o
    \\ \cU
                    \end{pmatrix} \to \begin{pmatrix}
                        \cX^o
                     \\ \cY
                    \end{pmatrix}$$
                    of
                    ${\Sigma_o}$  is unitary.  
                  Let $x\in \cX^o.$  Then, by contractivity of $T$ and unitarity of $T_{\Sigma_o}$
                \begin{align*}  
                 E \left(\begin{pmatrix}
                   A_{1} & A_{2} & B_{1} \\
                  0 & A_{o} & B_{o} \\
                  0 & C_{o} & D
                \end{pmatrix} \begin{pmatrix}
                   0\\ x\\
                    0
                  \end{pmatrix} \right)&=E \left(\begin{pmatrix}
                  A_{2} x\\ A_{o} x\\
                    C_{o} x
                  \end{pmatrix}\right)= E \left( A_{2} x\right) +  E \left(\begin{pmatrix}
                  A_{0} x\\
                    C_{o} x
                  \end{pmatrix}\right) \\ &=  E \left(  Tx\right) \leq
     E(
                    x)
           = E(  T_{\Sigma_o}x)    =
                E \left(\begin{pmatrix}
                  A_{0} x\\
                    C_{o} x
                  \end{pmatrix}\right).
                \end{align*} Since $A_{2} x \in (\cX^o)^\perp$ and $(\cX^o)^\perp$ is a Hilbert space, it follows that  $A_{2}  =0.$ If one chooses $u \in \cU,$ a similar argument as above shows that $B_1=0.$ Then for any $n \in \dN$, it holds \begin{align*}  A^nB&=  \begin{pmatrix}
                   A_{1} &0  \\
                  0 & A_{o}
                \end{pmatrix}^n  \begin{pmatrix}
                    0 \\
                   B_{o}
                \end{pmatrix}=\begin{pmatrix}
                    0 \\
                  A_o^n B_{o}
                \end{pmatrix}, \\ A^{*n}C^*&=  \begin{pmatrix}
                   A_{1}^* &0  \\
                  0 & A_{o}^*
                \end{pmatrix}^n  \begin{pmatrix}
                    0 \\
                   C_{o}^*
                \end{pmatrix}=\begin{pmatrix}
                    0 \\
                  A_o^{*n} C_{o}^*
                \end{pmatrix} .\end{align*} This is only possible if $(\cX^o)^\perp=0,$ since $\Sigma$ is simple. But then the systems $\Sigma_0$ and $\Sigma$ coincide, so the system $\Sigma$ is minimal conservative, and  \textrm{(b)} holds.

Now \textrm{(b)} trivially implies \textrm{(c)}, and
the fact that \textrm{(c)} implies \textrm{(a)} follows by combining the parts \textrm{(I)} and \textrm{(II)}. The proof is complete.
\end{proof}
\begin{remark}
  If $\cU$ and $\cY$ are Hilbert spaces, it follows from \cite[Lemma 3.2]{Lassi} that  simple passive realizations of $\theta \in \SK$ are $\kappa$-admissible. Therefore, in that case
   it is not necessary to assume the considered systems to be $\kappa$-admissible in   Lemma \ref{equivopti} and  Theorems  \ref{sepontulos} and \ref{kulmaminimal},  since the other assumptions already guarantee it. However, if $\cU$ and $\cY$ are Pontryagin spaces with the same negative index, it is not known that are all simple passive, or even all minimal passive, realizations of $\theta \in \SK$ $\kappa$-admissible.
\end{remark}

  If $\varphi_{\theta}\equiv0$  ($\psi_{\theta}\equiv0$), then  Theorem \ref{sepontulos} shows  that all $\kappa$-admissible minimal passive realizations of $\theta \in \SK$ are minimal isometric (co-isometric). In particular, they are controllable isometric (observable coisometric), and it follows from Theorem \ref{realz} that they are unitarily similar. This situation can occur also when the defect functions do not vanish identically.
In what follows, the range of $\varphi_{\theta}$ and the domain of ${\psi_{\theta}}$ will be denoted, respectively, by $\mathfrak{D}_{\varphi_{\theta}}$ and $\mathfrak{D}_{\psi_{\theta}}$.   In the  Hilbert space settings, it is well known \cite{Boikop1,Boikop2} that for a standard Schur function   $\theta \in \mathbf{S}(\cU,\cY), $ there exists a function $\chi_{\theta} \in L^\infty(\mathfrak{D}_{\psi_{\theta}},\mathfrak{D}_{\varphi_{\theta}})$ such that the function
\begin{equation}\label{kulmafunktio}
  \Theta(\zeta):=\begin{pmatrix}
    \theta(\zeta) & \psi_{\theta}(\zeta) \\
    \varphi_{\theta}(\zeta) & \chi_{\theta}(\zeta)
  \end{pmatrix}
\end{equation} has contractive values for a.e. $\zeta \in \dT.$ Under certain normalizing conditions for the functions $ \varphi_{\theta}$ and $ \psi_{\theta}$, the function $\chi_{\theta}$ is unique. In general, $\chi_{\theta}$ may has negative Fourier coefficients and therefore it is not a Schur function. In that case the function $\Theta$ in \eqref{kulmafunktio} is not a Schur function either. However, Arov and Nudelmann showed in \cite{ArNu1,ArNu2} that $\Theta$  is a Schur function if and only if all minimal passive realizations of $\theta$ are unitarily similar. This result  will be generalized to the indefinite settings in the following theorem. The proof uses optimal and $^*$-optimal realizations as in \cite{ArNu1,ArNu2}, but it is more elementary.

\begin{theorem}\label{kulmaminimal}
  Let $\theta \in \SK$, where $\cU$ and $\cY$ are Pontryagin spaces with the same negative index, and let $\varphi_\theta$ and $\psi_\theta$ be defect functions of $\theta.$ Then all $\kappa$-admissible minimal passive realizations of $\theta$ are unitarily similar if and only if there exist an $\cL(\mathfrak{D}_{\psi_{\theta}},\mathfrak{D}_{\varphi_{\theta}})$-valued function
  $\chi_{\theta}$  analytic in a neighbourhood of the origin such that
  \begin{equation}\label{kulmafunktio2}
  \Theta=\begin{pmatrix}
    \theta & \psi_{\theta} \\
    \varphi_{\theta}& \chi_{\theta}
  \end{pmatrix} \in \mathbf{S}_{\kappa}\left(\begin{pmatrix}
                  \cU \\ \sD_{\psi_{\theta}}
                 \end{pmatrix} , \begin{pmatrix}
                   \cY\\ \sD_{\varphi_{\theta}}
                 \end{pmatrix} \right) \end{equation}
\end{theorem}
\begin{proof}
  Suppose that all $\kappa$-admissible minimal passive realizations of $\theta \in \SK$ are unitarily similar. Then it  follows from Lemma \ref{equivopti} that  every $\kappa$-admissible minimal passive realization is optimal and $^*$-optimal. Take any $\kappa$-admissible minimal passive realization $\Sigma$ of $\theta$ and consider its Julia embedding as in \eqref{Juliaemb}. Then the transfer function \eqref{embtrans} of the Julia embedding belongs to the class   $\mathbf{S}_{\kappa}\left(
                  \cU \oplus \sD_{T^*}
                 ,
                   \cY \oplus \sD_{T}
                  \right),$ and since $\Sigma$ is both optimal and $^*$-optimal, the upper right corner and lower left corner of \eqref{embtrans} are defect functions of $\theta.$ Choose $\chi_{\theta}=\chi$ in \eqref{embtrans}, and the necessity is proven.

                  Suppose then that there exists  an $\cL(\mathfrak{D}_{\psi_{\theta}},\mathfrak{D}_{\varphi_{\theta}})$-valued function $\chi_\theta$ such that $\Theta$ in  \eqref{kulmafunktio2} belongs to the class  $\mathbf{S}_{\kappa}\left(
                  \cU \oplus\sD_{\psi_{\theta}}
                 ,
                   \cY \oplus \sD_{\varphi_{\theta}}
                  \right).$  It suffices to show that there exists  minimal passive realization $\Sigma$ of $\theta$ such that it is both optimal and $^*$-optimal; see Lemma \ref{equivopti}. Let $$\Sigma_{\Theta}=(A,\widetilde{B},\widetilde{C},\widetilde{D};\cX,  \cU \oplus\sD_{\psi_{\theta}} , \cY \oplus \sD_{\varphi_{\theta}};\kappa)$$ be a simple conservative realization of $\Theta \in \mathbf{S}_{\kappa}\left(
                  \cU \oplus\sD_{\psi_{\theta}}
                 ,
                   \cY \oplus \sD_{\varphi_{\theta}}
                  \right). $ Then the system operator $T_\Theta$ of $\Sigma_\Theta$  can be represented as
                  \begin{align*}
                    T_\Theta= \begin{pmatrix}
                                A & \begin{pmatrix}
                                      B & B_1
                                    \end{pmatrix} \\
                                \begin{pmatrix}
                                      C \\ C_1
                                    \end{pmatrix} & \begin{pmatrix}
                                      D & D_{12} \\
                                       D_{21} & D_{22}
                                    \end{pmatrix}
                              \end{pmatrix}: \begin{pmatrix}
                                               \cX \\
                                               \begin{pmatrix}
                                                \cU \\\sD_{\psi_{\theta}}
                                             \end{pmatrix}
                                             \end{pmatrix} \to \begin{pmatrix}
                                               \cX \\
                                                \begin{pmatrix}
                                                \cY \\\sD_{\varphi_{\theta}}
                                             \end{pmatrix}
                                             \end{pmatrix}. 
                  \end{align*} In a sufficiently small symmetric neighbourhood $\Omega$ of the origin,  it holds
                  \begin{align*}  \Theta(z)&=\begin{pmatrix}
                                   \theta(z) & \psi_{\theta}(z) \\
    \varphi_{\theta}(z) & \chi_{\theta}(z)
                                \end{pmatrix}  \\&=\begin{pmatrix}
                                      D+ zC(I-zA)^{-1}B & D_{12}+zC(I-zA)^{-1}B_1 \\
                                       D_{21}+zC_1+(I-zA)^{-1}B & D_{22}+zC_1(I-zA)^{-1}B_1  \end{pmatrix}.     \end{align*}
                                        The spaces $\sD_{\varphi_{\theta}}$ and $\sD_{\psi_{\theta}}$ are Hilbert spaces, and therefore it follows that $\sys$ is a passive realization of $\theta.$ Since $\Sigma_\Theta$ is conservative, Lemma \ref{ker-es} shows  
                                        that
                                       \begin{align*}
                                       &I-\Theta(z)\Theta^*(w)\\&= \begin{pmatrix}
                                                                  I_{\cY}-\theta(z)\theta^*(w)-\psi_{\theta}(z)\psi_{\theta}^*(w) & -\theta(z)\varphi_{\theta}^*(w)-\psi_{\theta}(z)\chi_{\theta}^*(w) \\
                                                                  -\varphi_{\theta}(z)\theta^*(w)-\chi_{\theta}(z)\psi_{\theta}^*(w) & I_{\sD_{\varphi_{\theta}}}- \varphi_{\theta}(z)\varphi_{\theta}^*(w)-\chi_{\theta}(z)\chi_{\theta}^*(w)
                                                                \end{pmatrix} \\&=(1-\bar{w}z) \widetilde{C}(I-zA)^{-1} (I-\bar{w}A^*)^{-1}\widetilde{C}^* \\
                                                                &=(1-\bar{w}z)\begin{pmatrix}
                                                                    C(I-zA)^{-1}(I-\bar{w}A^*)^{-1}C^* &  C(I-zA)^{-1}(I-\bar{w}A^*)^{-1}C_1^* \\
                                                                     C_1(I-zA)^{-1}(I-\bar{w}A^*)^{-1}C^* &  C_1(I-zA)^{-1}(I-\bar{w}A^*)^{-1}C_1^*
                                                                  \end{pmatrix}
                                                                 \\
                                          &I-\Theta^*(w) \Theta(z)\\&= \begin{pmatrix}
                                                              I_{\cU}-\theta^*(w)\theta(z)-\varphi_{\theta}^*(w)\varphi_{\theta}(z)   & -\theta^*(w) \psi_{\theta}(z) - \varphi_{\theta}^*(w)\chi_{\theta}(z)   \\
                                                              -\psi_{\theta}^*(w) \theta(z) -\chi_{\theta}^*(w) \varphi_{\theta}(z)   & I_{\sD_{\psi_{\theta}}}- \psi_{\theta}^*(w) \psi_{\theta}(z) - \chi_{\theta}^*(w) \chi_{\theta}(z)
                                                                \end{pmatrix}
                                          \\
                                                                &=(1-\bar{w}z)\!\begin{pmatrix}
                                                                    B^*(I-\bar{w}A^*)^{-1}(I-zA)^{-1}B &  B^*(I-\bar{w}A^*)^{-1}(I-zA)^{-1}B_1 \\
                                                                     B_1^*(I-\bar{w}A^*)^{-1}(I-zA)^{-1}B &  B_1^*(I-\bar{w}A^*)^{-1}(I-zA)^{-1}B_1
                                                                  \end{pmatrix}.
                                       \end{align*} That is,
                                       \begin{align}       I_{\cY}-\theta(z)\theta^*(w)&= (1-\bar{w}z)   C(I-zA)^{-1}(I-\bar{w}A^*)^{-1}C^*+\psi_{\theta}(z)\psi_{\theta}^*(w) \label{ekaaekv} , \\
                                                                I_{\cU}-\theta^*(w)\theta(z) &=  (1-\bar{w}z)B^*(I-\bar{w}A^*)^{-1}(I-zA)^{-1}B +\varphi_{\theta}^*(w)\varphi_{\theta}(z)       \label{tokaekv}   .\end{align} An easy calculation and Lemma \ref{related} show that the equation \eqref{ekaaekv} is equivalent to
                                                                \begin{align}  \label{kolmasekv}     
                                         I_{\cY}-{\theta^\#}^*(w)   \theta^\#(z)&\!=\! (1-\bar{w}z)   C(I-\bar{w}A)^{-1}(I-zA^*)^{-1}C^*+{ \varphi_{\theta^\#}}^*({w})    \varphi_{\theta^\#} ({z})
                                              .
                                             \end{align}
                                                                Let $$\Sigma'=(A',B',C',D;\cX',\cU,\cY;\kappa), \qquad  \Sigma''=(A'',B'',C'',D;\cX'',\cU,\cY;\kappa) 
                                                                $$
                                                                 be, respectively,
                                                                an  optimal minimal
                                                                 and  a $^*$-optimal minimal
                                                                  realizations of  $\theta$. It follows from Theorem \ref{optimals} (ii) that
                                                                  ${\Sigma''}^*=({A''}^*,{C''}^*,{B''}^*,D;\cX'',\cU,\cY;\kappa) $  is an optimal minimal realization of $\theta^\#.$
                                                                  Then, by the definition of $\varphi_{\theta}$ and   $\varphi_{\theta^\#},$ it holds
                                                             {\small     \begin{align*}
                                                                    I_{\cU}-\theta^*(w)\theta(z) &=  (1-\overline{w}z)B'^*(I-\bar{w}A'^*)^{-1}(I-zA')^{-1}B' +\varphi_{\theta}^*(w)\varphi_{\theta}(z) \\
                                                                      I_{\cY}-{\theta^\#}^*(w)   \theta^\#(z)&=  (1-\overline{w}z)   C''(I-\bar{w}A'')^{-1}(I-zA''^*)^{-1}C''^*+{ \varphi_{\theta^\#}}^*({w})    \varphi_{\theta^\#} ({z}).  \end{align*} }
                                                                  It follows that \begin{align*}  B^*(I-\bar{w}A^*)^{-1}(I-zA)^{-1}B&=  B'^*(I-\bar{w}A'^*)^{-1}(I-zA')^{-1}B',\\
                                                                     C(I-\bar{w}A)^{-1}(I-zA^*)^{-1}C^* &=  C''(I-\bar{w}A'')^{-1}(I-zA''^*)^{-1}C''^*.\end{align*}  By using Lemma \ref{suf-opti}, it can be deduced  that $\Sigma$ and $\Sigma^*$ are optimal systems.  Then it follows from Lemma \ref{opti-cont} that $\cX^c=\cX^o$ and therefore $\cX^s=\cX^c=\cX^o.$ By Lemma \ref{repres}, the restriction $\Sigma_s=(A_s,B_s,C_s,D;\cX^s,\cU,\cY;\kappa)$ of $\Sigma$ to the simple subspace $\cX^s$ is simple, and it holds $A^nB=A_s^nB_s$ and ${A^*}^nC^*={A_s^*}^nC_s^*$ for every $n \in \dN_0.$ That is, $\Sigma_s$ and $\Sigma_s^*$ also are optimal systems. Moreover, they are minimal  since $\cX^s=\cX^c=\cX^o.$ It follows now from Theorem \ref{optimals} (ii) that $\Sigma_s$ is also $^*$-optimal, and the proof is complete.
                                                                       \end{proof}

 \noindent \textbf{Acknowledgements}\quad I wish to thank  Seppo Hassi for helpful discussions  while preparing this paper.

\end{document}